\documentclass{elsarticle}

\usepackage[english]{babel}
\usepackage{dsfont} % \mathds
\usepackage{graphicx}
\usepackage{color}
\usepackage[hyperindex,breaklinks]{hyperref}
\usepackage{enumerate}
\usepackage{amssymb,empheq} % \upharpoonright
\usepackage{amsthm} % \newtheorem
\usepackage{cleveref}
\usepackage{xcolor}
\usepackage{subcaption}

\allowdisplaybreaks
\newtheorem{thm}{Theorem}[section]

\newtheorem{signHausthm}[thm]{Signed Hausdorff's Theorem}
\newtheorem{richthm}[thm]{Richter's Theorem}

\newtheorem{cor}[thm]{Corollary}

\theoremstyle{definition}

\newtheorem{dfn}[thm]{Definition}
\newtheorem{exm}[thm]{Example}

\theoremstyle{remark}
\newtheorem{rem}[thm]{Remark}

\newcommand{\exmsymbol}{\hfill$\circ$}

\newcommand{\cset}{\mathds{C}}

\newcommand{\nset}{\mathds{N}}

\newcommand{\rset}{\mathds{R}}

\newcommand{\diff}{\mathrm{d}}

\newcommand{\lin}{\mathrm{lin}\,}

\newcommand{\supp}{\mathrm{supp}\,}

\newcommand{\one}{\mathds{1}}

\newcommand{\cat}{\mathcal{C}}

\newcommand{\cF}{\mathcal{F}}

\newcommand{\cS}{\mathcal{S}}

\newcommand{\cX}{\mathcal{X}}

\newcommand{\cV}{\mathcal{V}}

\newcommand{\cY}{\mathcal{Y}}

\newcommand{\fA}{\mathfrak{A}}

\author{Philipp J.\ di Dio}
\address{Department of Mathematics and Statistics, University of Konstanz, Universit\"atsstra{\ss}e 10, D-78464 Konstanz, Germany}
\address{Zukunftskolleg, Universtity of Konstanz, Universit\"atsstra{\ss}e 10, D-78464 Konstanz, Germany}
\address{philipp.didio@uni-konstanz.de}

\journal{arXiv}

\title{Absolutely Continuous Representing Measures of Moment Sequences with an Emphasis on the Regularity of the Densities}

\begin{document}%%%
%%%%%%%%%%%%%%%%%%%

\begin{abstract}
In this work we investigate and characterize linear functionals $L:\rset[x_1,\dots,x_n]\to\rset$ with absolutely continuous representing measures $\mu$, i.e., $\diff\mu(x) = g(x)\,\diff x$ for some density $g$.
We focus on the regularity of $g$ and how it is determined and influenced by the moments $s_\alpha = L(x^\alpha)$.
\end{abstract}


\maketitle

\tableofcontents

\noindent\rule{\textwidth}{1pt}

\section{Introduction}%%%
%%%%%%%%%%%%%%%%%%%%%%%%%

Let $\mu$ be a measure on $\rset^n$ with $n\in\nset$.
For $\alpha = (\alpha_1,\dots,\alpha_n)\in\nset_0^n$ we call the number
\[s_\alpha := \int_{\rset^n} x^\alpha~\diff\mu(x)\]
with $x^\alpha := x_1^{\alpha_1}\cdots x_n^{\alpha_n}$ the \emph{$\alpha$-moment} of $\mu$ if it exists.
The \emph{$K$-moment problem} is: Given finitely or infinitely many real numbers $s_\alpha$, $\alpha\in I\subseteq\nset_0^n$, and a closed set $K\subseteq\rset^n$.
Does there exist a measure $\mu$ on $K$ such that
\[s_\alpha = \int_K x^\alpha~\diff\mu(x)\]
holds for all $\alpha\in I$?

For classical results on the moment problem see e.g.\ 
\cite{achieser56,ahiezer62,akhiezClassical,berg91,blekhe20,didio17w+v+,
didio18gaussian,didio21HilbertFunction,didio23gaussian,fialkoMomProbSurv,hamburger20,
hausdo21a,hausdo21,hausdo23,havila35,havila36,karlinStuddenTSystemsBook,kemper68,
kemper87,krein70,kreinMarkovMomentProblem,landauMomAMSProc,landau80,lauren09,
marshallPosPoly,schmudMomentBook,stielt94,simon98,strassen65} to name only a few references to demonstrate the rich and long history as well as the variety of applied techniques to this problem.

In applications moments and moment problems appear frequently, e.g.\
in financial mathematics \cite{anast06,stoyan16},
in (algebraic) statistics \cite{pearson94,titter85,martin05,amendo18,didio18gaussian},
in optimization \cite{lasserreSemiAlgOpt},
in shape reconstruction \cite{balins61,manas68,mathei80,lee82,milanf95,golub99,becker07,
sommer07,gravin12,gravin14,gravin18},
and in pattern recognition \cite{hu62,dai92,chen93,sommer07,ammari19}.
Especially for shape reconstruction and pattern recognition the concept of derivatives of moments and moment functionals was developed \cite{didio23gaussian}, see also \cite{lasserre08,henrion14,brehard19,marx20}.
In these applications moments of absolutely continuous measures are the main focus, i.e., measures $\mu$ of the form $\diff\mu(x) = g(x)\,\diff x$ for some density function $g$.
In \cite{ambros14} $L^1$-densities were treated.
In the current work we extend this work and look especially at the regularity of the density $g$.
We are here also facing the problem of signed representing measures, see e.g.\
\cite{borel95,polya38,boas39a,bloom53,sherman64,horn77,berg79,kowalski84,duran89,
hoischen92,schmud24signedArxiv,didio24tsystemshomepageArxiv}.

The paper is structured as follows.
In \Cref{sec:truncated} we treat the setup of \cite{ambros14} in full generality, i.e., truncated moment functionals and sequences ($I$ is finite resp.\ finite dimensional vector spaces $\cV$ of measurable functions) on general measurable spaces $(\cX,\fA)$.
In \Cref{thm:main} we give the main theorem which is based on/appeared in \cite{didio18gaussian}.
Its proof is significantly easier than the more special version in \cite{ambros14} since we use here the concept of Dirac approximating families (\Cref{dfn:dirac}) and \Cref{thm:richter} \cite{richte57}.
In \Cref{exm:discontinuous} we give an example of an inner moment functional which can not be represented by an absolutely continuous representing measure.
For the rest of \Cref{sec:truncated} we discuss and show how from \Cref{thm:main} the results in \cite{ambros14} follow and how \Cref{thm:main} also gives absolutely continuous representing measures for boundary moment functionals and on general measurable spaces $\cX$.
In \Cref{sec:full01} we then treat linear functionals $L:\rset[x_1,\dots,x_n]\to\rset$ such that the representing measures have a density $g$ with $\supp g \subseteq [0,1]^n$.
For that we employ the technique of derivatives of moments \cite{didio23gaussian} and the \Cref{thm:signedHaus} \cite{hausdo23}.
We characterize linear functionals $L$ having an absolutely continuous (signed) representing measure on $[0,1]^n$ and additionally also characterize them with given regularity (differentiability).
In \Cref{sec:fullR} we then treat linear functionals $L:\rset[x_1,\dots,x_n]\to\rset$ without any support restrictions.
Because of the flexibility of signed representing measures on all $\rset^n$ we have to restrict the treatment to sequences $s = (s_\alpha)_{\alpha\in\nset_0^n}$ such that
\[f(z) := \sum_{\alpha\in\nset_0^n} \frac{i^{|\alpha|}\cdot s_\alpha}{\alpha!}\cdot z^\alpha \tag{$z\in\rset^n$}\]
is entire.
The function $f$ turns out to be the characteristic function $\hat\mu$ of the representing measure $\mu$ and (inverse) Fourier transformation, which is well-known and used e.g.\ in signal processing, provides us with the density $g$ and hence characterizations of (signed) absolutely continuous representing measures.
In \Cref{sec:summary} we summarize the main results.

\section{Absolutely Continuous Representing Measures of Truncated Moment Functionals}%%%
%%%%%%%%%%%%%%%%%%%%%%%%%%%%%%%%%%%%%%%%%%%%%%%%%%%%%%%%%%%%%%%%%%%%%%%%%%%%%%%%%%%%%%%%
\label{sec:truncated}

\subsection{Main Result}

To formulate and prove the main results in this section about truncated moment functionals we need the following concept.

\begin{dfn}[{\cite[p.\ 3022, general setting (c)]{didio18gaussian}}]\label{dfn:dirac}
Let $(\cX,\fA)$ be a measurable space and let $\cV$ be a vector space of measurable functions $f:\cX\to\rset$. 
A family $\cF = \{\delta_{\sigma,x} \,|\, x\in\cX,\ \sigma\in\Lambda\}$, $\Lambda$ a set, of measures $\delta_{\sigma,x}$ on $\cX$ such that for any $x\in\cX$ there exists a sequence $(\sigma_{x,i})_{i\in\nset_0}\subseteq\Lambda$ with
\begin{equation}\label{eq:approx}
\lim_{i\to\infty} \int_\cX f(y)~\diff\delta_{\sigma_{x,i},x}(y) = f(x)
\end{equation}
for all $f\in\cV$ is called a \emph{Dirac-approximating family for $\cV$}.
\end{dfn}

Our guiding example of a Dirac-approximating family is the following which is also absolutely continuous with respect to the Lebesgue measure $\lambda$.

\begin{exm}\label{exm:gaussian}
Let $\Lambda = (0,1)$, let $\lambda$ be the Lebesgue measure on $\rset$, and let
$g_{\sigma,x}(y) := \frac{1}{\sqrt{2\pi \sigma^2}}\cdot e^{-\frac{(x-y)^2}{2\sigma^2}}$
be the Gaussian distribution centered at $x\in\rset$ with variance $\sigma\in\Lambda$.
Then we have
\[\lim_{\sigma\searrow 0} \int_\rset f(y)\cdot g_{\sigma,x}(y)~\diff\lambda(y) = f(x)\]
for all $f\in C_b(\rset,\rset)$ bounded continuous functions.
Therefore, $\{\delta_{\sigma,x} \,|\, \sigma\in (0,1),\ x\in\rset\}$ with $\delta_{\sigma,x}$ defined by
$\diff\delta_{\sigma,x}(y) := g_{\sigma,x}(y)\,\diff\lambda(y)$
is a Dirac-approximating family for $C_b(\rset,\rset)$.
\exmsymbol
\end{exm}

We now show that condition (\ref{eq:approx}) for Dirac-approximating families is crucial in the existence in absolutely continuous representing measures.

\begin{exm}\label{exm:discontinuous}
Let $\cX = [0,2]$, let $\cV = \lin\{f_0,f_1,f_2\}$ with $f_0(x)=1$, $f_1(x) =x$, and
\[f_2(x) := \begin{cases}
x^2 & \text{for}\ x\in [0,2]\setminus\{1\}\ \text{and}\\
-1 & \text{for}\ x=1.
\end{cases}\]
Then
\[\lim_{\sigma\searrow 0} \int_0^2 f_2(y)\cdot g_{\sigma,1}(y)~\diff\lambda(y) = 1 \neq -1 = f_2(1).\]
In fact, $\cV$ has no Dirac-approximation family which is absolutely continuous with respect to the Lebesgue measure $\lambda$ on $[0,2]$.
Hence, the linear functional $L:\cV\to\rset$ represented by $\lambda + \frac{8}{3}\delta_1$ is in the interior of the moment cone but can not be represented by an absolutely continuous measure with respect to $\lambda$.
To see this assume $\rho(x)\,\diff\lambda(x)$ with $\rho\in C([0,2],\rset)$ represents $L$.
Then
\[\int_0^2 f_2(x)\cdot \rho(x)~\diff\lambda(x) = L(f_2) = 0\]
implies $\supp\rho\subseteq\{0\}$ which implies $L(1) = 0$.
This is a contradiction since $L(1) >0$.
\exmsymbol
\end{exm}

We have seen that if (\ref{eq:approx}) does not hold with a family of absolutely continuous measures then we can construct a moment functional in the interior of the moment cone which has no absolutely continuous representing measure with respect to the Lebesgue measure.

We now prove that if there is at least one Dirac-approximating family of absolutely continuous measures with respect to the Lebesgue measure then any moment functional in the interior of the moment cone has an absolutely continuous representing measure.
To prove this result we need the following fundamental result due to H.\ Richter.

\begin{richthm}[{\cite[Satz 4]{richte57}}]\label{thm:richter}
Let $n\in\nset$, let $(\cX,\fA)$ be a measurable space, let $\cV$ be a $n$-dimensional real vector space of measurable functions $f:\cX\to\rset$, and let $L:\cV\to\rset$ be a linear functional given by some measure $\mu$ on $\cX$, i.e., $L(f) = \int_\cX f(x)~\diff\mu(x)$ for all $f\in\cV$.
Then there exist a $k\in\nset_0$ with $k\leq n$, points $x_1,\dots,x_k\in\cX$, and coefficients $c_1,\dots,c_k>0$ such that
\[L(f) = \sum_{i=1}^k c_i\cdot f(x_i)\]
holds for all $f\in\cV$.
\end{richthm}

For more on the history of Richter's Theorem see e.g.\ \cite{didioCone22}.
From Richter's Theorem we gain the following, which was first proven in \cite{didio18gaussian}.
We include the proof here to demonstrate how easy the existence of absolutely continuous representing measures follows from a Dirac-approximating family and Richter's Theorem.

\begin{thm}[{\cite[Thm.\ 17]{didio18gaussian}}]\label{thm:main}
Let $n\in\nset$,
let $(\cX,\fA)$ be a measurable space,
let $\cV$ be a $n$-dimensional vector space of measurable functions $f:\cX\to\rset$,
let $\cF=\{\delta_{\sigma,x} \,|\, \sigma\in\Lambda,\ x\in\cX\}$ be a Dirac-approximating family,
and let $L$ be a linear functional on $\cV$ from the interior of the moment cone.
Then there exist a $k\in\nset$ with $k\leq n$,
points $x_1,\dots,x_k\in\cX$, coefficients $c_1,\dots,c_k>0$, and parameters $\sigma_1,\dots,\sigma_k\in\Lambda$ such that
\[L(f) = \sum_{i=1}^k c_i\cdot\delta_{\sigma_i,x_i}\]
holds for all $f\in\cV$.
\end{thm}
\begin{proof}
Since $L$ is a moment functional by \Cref{thm:richter} there is a $k\in\nset$ with $k\leq n$, there are points $x_1,\dots,x_k\in\cX$, and $c_1',\dots,c_k'>0$ such that $L$ is represented by $\sum_{i=1}^k c_i' \delta_{x_i}$.
Since $L$ is in the interior of the moment cone we can assume $k=n$, i.e., $L$ is in the interior of the simplicial cone spanned by $l_i:\cV\to\rset$ which are represented by $\delta_{x_i}$ for all $i=1,\dots,n$.

Since $\cF$ is a Dirac-approximating family there exist $\sigma_1,\dots,\sigma_n\in\Lambda$ such that $L$ is in the interior of the simplicial cone spanned by $L_i$ represented by $\delta_{\sigma_i,x_i}$, $i=1,\dots,n$, which proves the assertion.
\end{proof}

The proof of \Cref{thm:main} is considerably easier and shorter than the proof of Theorem 2.6 in \cite{ambros14}.
When we take in \Cref{thm:main} the family $\cF$ to be absolutely continuous with respect to the Lebesgue measure $\lambda$ on $\rset^n$ (e.g.\ by \Cref{exm:gaussian}) then we immediately get the main result of \cite[Thm.\ 2.6]{ambros14}.

\begin{cor}[{\cite[Thm.\ 2.6]{ambros14}}]\label{cor:ambrosie}
Let $n\in\nset$, let $\lambda$ be the Lebesgue measure on $\rset^n$, let $\cX\subseteq\rset^n$ be a closed set such for any point $x\in\cX$ we have $\lambda(\cX\cap B_\varepsilon(x))>0$ for all $\varepsilon>0$ and $x\in\cX$ where $B_\varepsilon(x)$ is the ball around $x$ of radius $\varepsilon$, let $I\subsetneq\nset_0^n$ be finite with $0\in I$, let $\cV = \lin\{x^\alpha \,|\, \alpha\in I\}$, and let $L:\cV\to\rset$ be a linear functional.
The following are equivalent:
\begin{enumerate}[(i)]
\item $L$ is in the interior of the moment cone.

\item There exists a $\rho\in L^1(\cX,\lambda)$ with $\rho\geq 0$ such that all $f\in\cV$ are $\rho~\diff\lambda$-integrable and
\[L(f) = \int_\cX f(x)\cdot \rho(x)~\diff\lambda(x)\]
holds for all $f\in\cV$, i.e., $L$ is represented by $\mu$ on $\cX$ with $\diff\mu = \rho~\diff\lambda$.
\end{enumerate}
\end{cor}

We will now discuss in which ways \Cref{thm:main} is stronger than \Cref{cor:ambrosie} and show how \Cref{thm:main} applies to moment functionals on the boundary of the moment cone.

\subsection{Density $\rho$ can be chosen to be $C^\infty$ and even $C^\infty_c$}
From \Cref{exm:gaussian} we see that in fact in \Cref{cor:ambrosie} we can choose $\rho$ to be a Schwartz function.
We can even choose $\rho$ to have compact support when the Dirac-approximating family $\cF$ does not come from the Gaussian distribution but from $C^\infty_c$-functions, e.g., a molifier.
We can also generate $\cF$ from characteristic functions $\chi_{B_\varepsilon(x)\cap\cX}$.

\begin{cor}\label{cor:densityExtension}
If in \Cref{thm:main} we additionally have that the Dirac-approximating family $\cF$ is absolutely continuous (with $L^1$-, $C^\infty$-, or $C_c^\infty$-densities) with respect to a measure $\mu$ on $\cX$ then also $L$ has an absolutely continuous representing measure (with a $L^1$-, $C^\infty$-, or $C_c^\infty$-density, respectively) with respect to $\mu$ on $\cX$.
\end{cor}

Dirac-approximation families with $C^\infty$- or even with $C_c^\infty$-densities always exists for the cases covered by \Cref{cor:ambrosie}.
Take the molifier $S_\varepsilon$ as the densities with suitable normalization constants.

\subsection{Extension of the function space $\cV$}
We have that $0\in I$, i.e., $1\in\cV$, is not required in \Cref{thm:main} since it is not required in \Cref{thm:richter}.
Additionally, when $\cV$ are ``only'' continuous functions instead of polynomials then \Cref{exm:gaussian} always gives a Dirac-approximating family.
If we have a function $f\in\cV$ which is discontinuous then \Cref{exm:discontinuous} provides an example where no absolutely continuous representing measure exists, at least not with respect to the Lebesgue measure.

\subsection{Density $\rho$ can be chosen with respect to more than the Lebesgue measure}
While in \Cref{exm:discontinuous} we can not get an absolutely representing measure with respect to the Lebesgue measure $\lambda$, we get absolutely representing measures with respect to the measure $\mu = \lambda + \delta_1$.
To see this we have to construct a Dirac-approximating family $\cF$ which is absolutely continuous with respect to $\mu$.
If $x\in [0,2]\setminus\{1\}$ we take for every $\varepsilon>0$ a function $f_{\varepsilon,x}\in C_c^\infty([0,2],\rset)$ such that $f_{\varepsilon,x}(1)=0$, $f_{\varepsilon,x}\geq 0$, $\supp f_{\varepsilon,x}\subseteq [x-\varepsilon,x+\varepsilon]$, and $\int_0^2 f_{\varepsilon,x}(y)~\diff\lambda(y) = 1$.
Define $\delta_{\varepsilon,x}$ by $\diff\delta_{\varepsilon,x} := f_{\varepsilon,x}~\diff\mu$. Continuity on $[0,1)\cup (1,2]$ then gives (\ref{eq:approx}).
If $x=1$ then take for every $\varepsilon>0$ a function $f_{\varepsilon,1}\in C_c^\infty([0,2],\rset)$ with $f_{\varepsilon,1}(1) = 1$, $0\leq f_{\varepsilon,x}\leq 1$, and $\supp f_{\varepsilon,1} \subseteq [1-\varepsilon,1+\varepsilon]$.
Then (\ref{eq:approx}) also holds.
This family $\cF$ provides an absolutely continuous representing measure with respect to $\mu = \lambda + \delta_1$ with density $\rho\in C^\infty([0,2],\rset)$.
That is summarized already in \Cref{cor:densityExtension}.

\subsection{Extension of $\cX$}
That $\cX$ is closed in \Cref{cor:ambrosie} is not necessary.
Additionally, general $\cX$ can be used, especially when the Lebesgue measure $\lambda$ on $\rset^n$ is replaced by other measures on $\cX$.
For example let $\cX = \cX_1 \cup \cX_2\subseteq\rset^2$ with $\cX_1 = [-1,0]\times \{0\}$ and $\cX_2 = [0,1)\times [-1,1]$.
Take on $\cX_1$ the $1$-dimensional Lebesgue measure $\lambda_1$, i.e., $\lambda_1([a,b]\times\{0\}) = b-a$ for all $-1\leq a \leq b \leq 0$, and take on $\cX_2$ the Lebesgue measure $\lambda_2$ on $\rset^2$.
Then $\lambda = \lambda_1 + \lambda_2$ is well-defined on $\cX = \cX_1\cup\cX_2$ since $\cX_1\cap\cX_2 = \{(0,0)\}$ and $\lambda_1(\{(0,0)\}) = \lambda_2(\{(0,0)\}) = 0$.
An absolutely continuous family $\cF$ on $\cX$ with respect to $\lambda$ can then be constructed similarly to the previous paragraph by using appropriate densities in $C^\infty(\cX,\rset)$.

We can also enforce that $\lambda$ has some mass on subsets of $\cX$, e.g., by writing $\cX = [0,1]\times [-1,1] = \cX_1\cup\cX_2$ with $\cX_1 = \{0\}\times [-1,1]$ and $\cX_2 = (0,1]\times [-1,1]$.
Then on $\cX_1$ we can take the $1$-dimensional Lebesgue measure and on $\cX_2$ the $2$-dimensional Lebesgue measure.
Single points $\{x\}$ must be treated with the Dirac measure $\delta_x$ in a similar way as the absolutely continuous measure was constructed in the previous paragraph with respect to $\mu = \lambda + \delta_1$.

This procedure can be extended straight forward to general $\cX = \bigcup_{i\in I} \cX_i$, $I$ some index set, which is composed of e.g.\ open manifolds $\cX_i$ in $\rset^n$ or abstract measurable sets $\cX_i$.
Important is that $\cV$ are measurable functions on $\cX$ and on each $\cX_i$ we take a measure $\lambda_i$ such that we find an absolutely continuous Dirac-approximation family $\cF$ with respect to $\lambda = \sum_{i\in I} \lambda_i$.
\Cref{thm:main} then provides an absolutely continuous representing measure for all moment functional in the interior of the moment cone.
All that is already summarized in \Cref{cor:densityExtension}.

\subsection{Absolutely continuous representing measures for boundary moment functionals}
If $L$ lies on the boundary of the moment cone then by \cite[Prop.\ 7]{didio17Cara} we can restrict $\cX$ to some $\cY\subsetneq\cX$ such that $\cV|_\cY$ spans the unique face of the moment cone which contains $L$ in its relative interior.
Hence, $L$ is in the (relative) interior of a (smaller) moment cone spanned by $\cV|_\cY$ where again \Cref{thm:main} applies.

\begin{cor}
Let $n\in\nset$, let $(\cX,\fA)$ be a measurable space, and let $\cV$ be a $n$-dimensional real vector space of measurable functions $f:\cX\to\rset$.
If $L:\cV\to\rset$ is a moment functional on the boundary of the moment cone then there exists a $\cY\subsetneq\cX$ such that $L$ lies in the relative interior of the moment cone generated by $\cV|_\cY$ with dimension $\dim\cV|_\cY < n$.
If additionally there exists a measure $\mu$ on $\cY$ and an absolutely continuous Dirac-approximating family $\cF$ for $\cV|_\cY$ (with $L^1$-, $C^\infty$-, or $C_c^\infty$-densities) with respect to $\mu$ then $L$ has an absolutely continuous representing measure (with $L^1$-, $C^\infty$- or $C_c^\infty$-density) with respect to $\mu$ on $\cY$.
\end{cor}

\section{Absolutely Continuous Representing Measures of Full (Signed) Moment Functionals: Compact Case on $[0,1]^n$}%%%
%%%%%%%%%%%%%%%%%%%%%%%%%%%%%
\label{sec:full01}

We have seen in the previous section how easy the truncated case can be solved once one introduces the concept of Dirac approximating sequences in \Cref{dfn:dirac}.
We will now treat the much more complicated case of full moment functionals.

For this section we remind the reader of the concept of derivatives of moments and moment functionals which was introduced in \cite{didio23gaussian}.
For a moment sequence $s = (s_\alpha)_{\alpha\in\nset_0^n}$ with representing measure $\mu$, i.e.,
\[s_\alpha = \int_{\rset^n} x^\alpha~\diff\mu(x)\]
for all $\alpha\in\nset_0^n$, we define the $\partial^\beta$-derivative $\partial^\beta s = (\partial^\beta s_\alpha)_{\alpha\in\nset_0^n}$ in the distributional sense
\[\partial^\beta s_\alpha := (-1)^{|\beta|}\cdot\int_{\rset^n} \partial^\beta x^\alpha~\diff\mu(x)\]
for all $\alpha,\beta\in\nset_0^n$.
In the distributional sense we then have that $\partial^\beta s$ is represented by $\partial^\beta\mu$.

We deal here with the full (signed) moment problem on $[0,1]^n$.
To solve it we need the following.

\begin{signHausthm}[{\cite[p.\ 232, II.]{hausdo23}} or see e.g.\ {\cite[Thm.\ 3.3.1]{lorentz86}}]\label{thm:signedHaus}
Let $(s_i)_{i\in\nset_0}\subseteq\rset$ be a real sequence.
The following are equivalent:
\begin{enumerate}[(i)]
\item There exist positive ($\cat([0,1],\rset)$-regular) measures $\mu_1$ and $\mu_2$, i.e., a signed ($\cat([0,1],\rset)$-regular) measure $\mu = \mu_1 - \mu_2$, such that
\[s_i = \int_0^1 x^i~\diff\mu_1(x) - \int_0^1 x^i~\diff\mu_2(x) = \int_0^1 x^i~\diff\mu(x)\]
holds for all $i\in\nset_0$.

\item There exists a $C>0$ such that
\[\sum_{k=0}^d \binom{d}{k}\cdot \left|L_s(x^k\cdot (1-x)^{d-k})\right| < C\]
holds for all $d\in\nset_0$.
\end{enumerate}
\end{signHausthm}

The previous result also holds on $[0,1]^n$ for any $n\in\nset_0$ where (ii) becomes
\begin{enumerate}[(i')]\setcounter{enumi}{1}
\item There exists a $C>0$ such that
\[\sum_{k_1=0}^{d_1} \dots \sum_{k_n=0}^{d_n} \binom{d_1}{k_1}\cdots\binom{d_n}{k_n}\cdot \left|L_s(x_1^{k_1}\cdots x_n^{k_n}\cdot (1-x_1)^{d_1-k_1}\cdots (1-x_n)^{d_n-k_n})\right| < C\]
holds for all $d_1,\dots,d_n\in\nset_0$.
\end{enumerate}

The \Cref{thm:signedHaus} provides us with $C([0,1],\rset)$-regular measures $\mu_1$ and $\mu_2$.
By the Lebesgue decomposition \cite[Thm.\ 3.2.3]{bogachevMeasureTheory} we have that $\mu_i = f_i\cdot\lambda + \mu_{i,0}$ where $\mu_{i,0}$ is an atomic measure.
In the distributional sense $\delta_x'$ is no $C([0,1],\rset)$-regular measure.
We can therefore exclude $\delta_x$ and hence $\mu_{i,0}$ by applying the \Cref{thm:signedHaus} to $\partial s$.

Hence, by using the \Cref{thm:signedHaus} we get the following.
Derivatives $f'$ are meant always as distributional derivatives.
If $f'$ exists in the classical sense both coincide of course, see e.g.\ \cite{grubbDistributions} for more.

\begin{thm}\label{thm:AbsCont}
Let $s = (s_i)_{i\in\nset_0}$ be a real sequence.
The following are equivalent:
\begin{enumerate}[(i)]
\item There exists a piece-wise continuous function $f:[0,1]\to\rset$ such that
\[s_i = \int_0^1 x^i\cdot f(x)~\diff x\]
holds for all $i\in\nset_0$ and $f'$ is a $C([0,1],\rset)$-regular signed measure.

\item There exists a constant $C>0$ such that
\begin{equation}\label{eq:thmAbsCont1a}
\sum_{k=0}^d \binom{d}{k}\cdot \left|  L_s\big( x^k\cdot (1-x)^{d-k} \big) \right| \;<\; C
\end{equation}
and
\begin{equation}\label{eq:thmAbsCont1b}
\sum_{k=0}^d \binom{d}{k}\cdot \left|  L_{\partial s}\big( x^k\cdot (1-x)^{d-k} \big) \right| \;<\; C
\end{equation}
hold for all $d\in\nset_0$.
\end{enumerate}
If additionally $L_s\big(x^k\cdot (1-x)^l\big) \geq 0$ holds for all $k,l\in\nset_0$ then $f\geq 0$ on $[0,1]$.
\end{thm}
\begin{proof}
(i) $\Rightarrow$ (ii):
Since $f$ is piece-wise continuous we have (\ref{eq:thmAbsCont1a}) by the \Cref{thm:signedHaus}.
Since $f$ is piece-wise continuous we have that $f'$ exists in the distributional sense and it is a signed measure \cite{grubbDistributions}, i.e., also (\ref{eq:thmAbsCont1b}) holds by the \Cref{thm:signedHaus}.
Taking $C>0$ to be the maxima of the constants from (\ref{eq:thmAbsCont1a}) and (\ref{eq:thmAbsCont1b}) proves (ii).

(ii) $\Rightarrow$ (i):
By (\ref{eq:thmAbsCont1a}) we have that there exists a $C([0,1],\rset)$-regular measure $\nu$ representing $s$ and by (\ref{eq:thmAbsCont1b}) there exists a $C([0,1],\rset)$-regular measure $\mu$ representing $\partial s$.
But $\nu'$ exists in the distributional sense \cite{grubbDistributions} and hence by the Stone--Weierstra{\ss} Theorem we have $\nu' = \mu$ since $[0,1]$ is compact.
Hence, $f(x) := \int_{-\infty}^x 1~\diff\mu(y)$ is piece-wise continuous by the Lebesgue decomposition \cite[Vol.\ 1, Thm.\ 3.2.3]{bogachevMeasureTheory}.
Hence, $\diff\nu(x) = f~\diff x$ and therefore $\supp f \subseteq [0,1]$ which proves (i).

If additionally $L_s\big(x^k\cdot (1-x)^l\big) \geq 0$ holds for all $k,l\in\nset_0$ then $f~\diff x$ not only uniquely represents $s$ but is also a moment measure, i.e., $f\geq 0$.
\end{proof}

Unfortunately, \Cref{thm:AbsCont} does not cover all cases of absolutely continuous measures on $[0,1]^n$ as the following examples show.

\begin{exm}\label{exm:counter1}
Let $f(x) = x^{-1/4}$ be on $(0,1]$ and $0$ everywhere else.
Then $f$ is the density of a $C([0,1],\rset)$-regular measure but $f'(x) = -\frac{1}{4}\cdot x^{-5/4}$ on $(0,1)$ is no longer the density of a $C([0,1],\rset)$-regular measure.
\exmsymbol
\end{exm}

The reason that the previous example fails is of course the infinite jump at $x=0$.
But also other effects can appear as the following example shows.

\begin{exm}\label{exm:counter2}
Let $f(x) = x\cdot \sin(x^{-2})$ on $[0,1]$.
Then $f$ is the density of a $C([0,1],\rset)$-regular measure but $f'(x) = \sin(x^{-2}) -2\cdot x^{-2}\cdot \sin(x^{-2})$ on $(0,1)$ is not a density of a $C([0,1],\rset)$-regular measure.
\exmsymbol
\end{exm}

The condition that $f'$ is $C([0,1],\rset)$-regular in \Cref{thm:AbsCont} (i) is therefore really necessary.
It excludes certain sequences $s$ which can be represented by absolutely continuous measures.
But cases as in \Cref{exm:counter1} and \ref{exm:counter2} are covered by the $L^1$- resp.\ $L^2$-treatment in \Cref{sec:fullR}.

The function
\begin{equation}\label{eq:distributionFunction}
\mu\big((-\infty,x]\big) = \int_{-\infty}^x 1~\diff\mu(x)
\end{equation}
is called the \emph{distribution function} of $\mu$, compare with \cite[Vol.\ 1, Ch.\ 1.8]{bogachevMeasureTheory}.
Note the difference, that in \cite[Vol.\ 1, Ch.\ 1.8]{bogachevMeasureTheory} the distribution function is $\mu((-\infty,x))$.
Since $\mu((-\infty,x])$ also exists in the multi-dimensional case
\[\mu\big((-\infty,x_1]\times\dots\times (-\infty,x_n]\big) = \int_{-\infty}^{x_1} \dots \int_{-\infty}^{x_n} 1~\diff\mu(x_1,\dots,x_n)\]
for all $x=(x_1,\dots,x_n)^T\in\rset^n$ with $n\in\nset$ and the \Cref{thm:signedHaus} also holds on $[0,1]^n$ we immediately get the multi-dimensional version of \Cref{thm:AbsCont} in the following which we only state as a corollary of \Cref{thm:AbsCont}.

\begin{cor}\label{cor:absContMultiDim}
Let $n\in\nset$ and let $s = (s_\alpha)_{\alpha\in\nset_0^n}$ be a real sequence.
Then the following are equivalent:
\begin{enumerate}[(i)]
\item There exists an almost everywhere continuous function $f:[0,1]^n\to\rset$ such that
\[s_\alpha = \int_{[0,1]^n} x^\alpha\cdot f(x)~\diff x\]
holds for all $\alpha\in\nset_0^n$ and $\partial_1\dots\partial_n f$ is a $C([0,1]^n,\rset)$-regular signed measure.

\item There exists a constant $C>0$ such that
\begin{multline*}
\sum_{k_1=0}^{d_1} \dots \sum_{k_n=0}^{d_n} \binom{d_1}{k_1}\cdots \binom{d_n}{k_n}\\ \times \left| L_s\big( x^{k_1}\cdots x^{k_n}\cdot (1-x_1)^{d_1-k_1}\cdots (1-x_n)^{d_n-k_n} \big) \right| < C
\end{multline*}
and
\begin{multline*}
\sum_{k_1=0}^{d_1} \dots \sum_{k_n=0}^{d_n} \binom{d_1}{k_1}\cdots \binom{d_n}{k_n}\\ \times \left|  L_{\partial_1\dots\partial_n s}\big( x^{k_1}\cdots x^{k_n}\cdot (1-x_1)^{d_1-k_1}\cdots (1-x_n)^{d_n-k_n} \big) \right| < C
\end{multline*}
hold for all $d_1,\dots,d_n\in\nset_0$.
\end{enumerate}
If additionally
\[L_s\big( x_1^{k_1}\cdots x_n^{k_n}\cdot (1-x_1)^{d_1-k_1}\cdots (1-x_n)^{d_n-k_n}\big) \geq 0\]
holds for all $k_1,\dots,k_n,l_1,\dots,l_n\in\nset_0$ then $f\geq 0$ on $[0,1]^n$.
\end{cor}
\begin{proof}
Adapt the proof of \Cref{thm:AbsCont}.
\end{proof}

We have seen that (\ref{eq:distributionFunction}) is a piece-wise continuous function since $\mu$ is $C([0,1],\rset)$-regular, see also \cite[Thm.\ 1.8.1]{bogachevMeasureTheory} for the special case $\mu\geq 0$.
Hence, its integral (anti-derivative)
\[f^{(-1)}(x) := \int_{-\infty}^x \mu\big((-\infty,y]\big)~\diff y\]
is continuous ($f^{(-1)}\in C(\rset,\rset)$) with $f(x) = 0$ for all $x\leq 0$ and $f(x) = \mu(\rset)$ for all $x\geq 1$.
Continuing in this way we define
\begin{equation}\label{eq:antiDeriv}
f^{(-r-1)}(x) := \int_{-\infty}^x f^{(-r)}(y)~\diff y \qquad\in C^{r-2}(\rset,\rset)
\end{equation}
and get $f^{(-r)}(x) = 0$ for all $x\leq 0$ and $r\in\nset$.

Therefore, let $t = (t_i)_{i\in\nset_0}$ be a sequence and set $s = \partial^{r+1} t$ in \Cref{thm:AbsCont} then we get a piece-wise continuous representing measure of $s = \partial^{r+1} t$ and by integration a $C^r$-absolutely continuous representing measure of $t$.

\begin{cor}\label{cor:CrAbsCont}
Let $s=(s_i)_{i\in\nset_0}$ be a real sequence and let $r\in\nset_0$.
Then the following are equivalent:
\begin{enumerate}[(i)]
\item There exists a $C^r$-function $f:\rset\to\rset$ such that
\[f^{(j)}(x) = 0\]
holds for all $j=0,1,\dots,r$ and $x\leq 0$, $f^{(r+2)}$ is a $C([0,1],\rset)$-regular signed measure, and
\[s_i = \int_0^1 x^i\cdot f(x)~\diff x\]
holds for all $i\in\nset_0$.

\item There exists a constant $C>0$ such that
\[\sum_{k=0}^d \binom{d}{k}\cdot \left|  L_{\partial^{r+1} s}\big( x^k\cdot (1-x)^{d-k} \big) \right| < C\]
and
\[\sum_{k=0}^d \binom{d}{k}\cdot \left|  L_{\partial^{r+2} s}\big( x^k\cdot (1-x)^{d-k} \big) \right| < C\]
hold for all $d\in\nset_0$.
\end{enumerate}
If additionally $L_s \big(x^k\cdot (1-x)^l\big) \geq 0$ holds for all $k,l\in\nset_0$ then $f\geq 0$ on $[0,1]$.
\end{cor}
\begin{proof}
Use $\partial^{r+1} s$ instead of $s$ in \Cref{thm:AbsCont} and calculate (\ref{eq:antiDeriv}), i.e., $\partial^{r+1} s$ has a piece-wise continuous representing measure.
By integration $\partial^r s$ has a $C^0$-continuous representing measure which is zero at $x=0$ and proceeding by induction with (\ref{eq:antiDeriv}) proves the equivalence.
\end{proof}

It is easy to see that also \Cref{cor:CrAbsCont} has a multi-dimensional version which is similar to \Cref{cor:absContMultiDim}.

In \Cref{cor:CrAbsCont} the $C^r$ function $f$ must fulfill $f^{(j)}(0) = 0$ for all $j=0,1,\dots,r$.
General $f\in C^r([0,1],\rset)$ are not covered by \Cref{cor:CrAbsCont}.
Given a general $f\in C^r([0,1],\rset)$ then we can write it as
\[ f = f_1 + f_2\]
with $f_1^{(j)}(0) = 0$ and $f_2^{(j)}(1)=0$ for all $j=0,1,\dots,r$, e.g., $f_1|_{[0,1/4]}=0$ and $f_2|_{[3/4,1]}=0$.
Mirroring $f_2(x)$ to $f_2(-x)$ and shifting gives a function $\tilde{f}_2$ as in \Cref{cor:CrAbsCont} (i).
We switch from $[0,1]$ to $\big[-\frac{1}{2},\frac{1}{2}\big]$ to avoid shifting.

\begin{thm}\label{cor:CrAbsContGeneral}
Let $n\in\nset$, let $s = (s_\alpha)_{\alpha\in\nset_0^n}$ be a real sequence, and let $r\in\nset_0$.
Then the following are equivalent:
\begin{enumerate}[(i)]
\item There exists a function $f:\big[-\frac{1}{2},\frac{1}{2}\big]^n\to\rset$ such that
\begin{enumerate}[(a)]
\item $\displaystyle s_\alpha = \int_{\big[-\frac{1}{2},\frac{1}{2}\big]^n} x^\alpha\cdot f(x)~\diff x$ holds for all $\alpha\in\nset_0^n$,

\item $\partial_1^{\alpha_1}\dots\partial_n^{\alpha_n} f \in C^0\big( \big[-\frac{1}{2},\frac{1}{2}\big]^n,\rset\big)$ for all $\alpha_1,\dots,\alpha_n = 0,\dots,r$, and

\item $\partial_1^{r+2}\dots\partial_n^{r+2} f$ is a $C^0\big( \big[-\frac{1}{2},\frac{1}{2}\big]^n,\rset\big)$-regular signed measure.
\end{enumerate}

\item There exist $2^n$ real sequences $s^\sigma = (s_\alpha^\sigma)_{\alpha\in\nset_0^n}$ with $\sigma\in \{-1,1\}^n$ and a constant $C>0$ such that
\begin{enumerate}[(a)]
\item $\displaystyle s = \sum_{\sigma\in\{-1,1\}^n} s^\sigma$ and

\item for all $\sigma\in\{-1,1\}^n$ the sequence $\tilde{s}^\sigma := \big((-1)^{\frac{1}{2}\langle\alpha,\one-\sigma\rangle}\cdot s_\alpha^\sigma\big)_{\alpha\in\nset_0^n}$ fulfills
\begin{multline*}
\sum_{k_1=0}^{d_1}\cdots \sum_{k_n=0}^{d_n} \binom{d_1}{k_1}\cdots\binom{d_n}{k_n} \left|L_{\partial_1^{r+1}\dots \partial_n^{r+1} \tilde{s}^\sigma} \left( \left(\frac{1}{2}+x_1\right)^{k_1}\cdots \left(\frac{1}{2}+x_n\right)^{k_n}\right.\right.\\
\times \left.\left. \left(\frac{1}{2}-x_1\right)^{d_1-k_1}\cdots \left(\frac{1}{2}-x_n\right)^{d_n-k_n} \right)\right| < C
\end{multline*}
and
\begin{multline*}
\sum_{k_1=0}^{d_1}\cdots \sum_{k_n=0}^{d_n} \binom{d_1}{k_1}\cdots\binom{d_n}{k_n} \left|L_{\partial_1^{r+2}\dots \partial_n^{r+2} \tilde{s}^\sigma} \left( \left(\frac{1}{2}+x_1\right)^{k_1}\cdots \left(\frac{1}{2}+x_n\right)^{k_n}\right.\right.\\
\times \left.\left. \left(\frac{1}{2}-x_1\right)^{d_1-k_1}\cdots \left(\frac{1}{2}-x_n\right)^{d_n-k_n} \right)\right| < C
\end{multline*}
for all $d_1,\dots,d_n\in\nset_0$.
\end{enumerate}
\end{enumerate}
The function $f$ in (i) is unique.
If additionally
\begin{equation}\label{eq:Lpos}
L_s\left(\left(\frac{1}{2} + x_1\right)^{k_1}\cdots\left(\frac{1}{2} + x_n\right)^{k_n}\cdot \left(\frac{1}{2}-x_1\right)^{l_1}\cdots\left(\frac{1}{2}-x_n\right)^{l_n}\right) \geq 0
\end{equation}
holds for all $k_1,\dots,k_n,l_1,\dots,l_n\in\nset_0$ then $f\geq 0$ in (i).
\end{thm}
\begin{proof}
It is sufficient to prove the result for $n=1$.
For $n\geq 2$ the same decomposition and mirroring arguments hold.
Here, $\sigma = (\sigma_1,\dots,\sigma_n)$, i.e., the $\sigma_i =-1$, indicate the directions $x_i$ in which $x_i\mapsto -x_i$ has to be applied.

(i) $\Rightarrow$ (ii):
Write $f = f_1 + f_{-1}$ with $f_1,f_{-1}\in C^r\big(\big[-\frac{1}{2},\frac{1}{2}\big],\rset\big)$ and $f_1|_{[-1/2,-1/4]}=0$ and $f_{-1}|_{[1/4,1/2]} = 0$.
Set
\[u_i = s_i^{(1)} = \int_{-\frac{1}{2}}^{\frac{1}{2}} x^i\cdot f_1(x)~\diff x \qquad\text{and}\qquad v_i = s_i^{(-1)} = \int_{-\frac{1}{2}}^{\frac{1}{2}} x^i\cdot f_2(x)~\diff x\]
for all $i\in\nset_0$.
Then
\[(-1)^i\cdot v_i = \tilde{s}_i^{(-1)} = \int_{-\frac{1}{2}}^{\frac{1}{2}} x^i\cdot f_2(-x)~\diff x\]
holds for all $i\in\nset_0$ which proves (ii) by the \Cref{thm:signedHaus}.

(ii) $\Rightarrow$ (i):
Let $f_1$ for $u$ and $\tilde{f}_{-1}$ for $\tilde{v}$ be from \Cref{cor:CrAbsCont}.
Then
\[v_i = \int_{-\frac{1}{2}}^{\frac{1}{2}} x^i\cdot \tilde{f}_{-1}(-x)~\diff x\]
holds for all $i\in\nset_0$.
Hence, $f(x) := f_1(x) + \tilde{f}_{-1}(-x)$ is $C^r$ and represents $s$.

Uniqueness of $f$ holds since $\rset[x_1,\dots,x_n]$ is dense in $C^0\big( \big[-\frac{1}{2},\frac{1}{2}\big]^n,\rset\big)$, also for signed measures.
Since the signed representing measure of $L_s$ is unique and $f~\diff^n x$ is a representing measure we have that (\ref{eq:Lpos}) implies $f\geq 0$.
\end{proof}

\section{Absolutely Continuous Representing Measures of Full (Signed) Moment Functionals: Non-Compact Case on $\rset^n$}%%%
%%%%%%%%%%%%%%%%%%%%%%%%%%%%%%%%%
\label{sec:fullR}

The case of full moment functionals $L:\rset[x_1,\dots,x_n]\to\rset$ on $\rset^n$ is the most difficult.
The reason is the huge flexibility on $\rset^n$, i.e., for every sequences $s=(s_\alpha)_{\alpha\in\nset_0^n}$ there exists a signed measure $\mu$ such that
\[s_\alpha = \int_{\rset^n} x^\alpha~\diff\mu(x)\]
holds for all $\alpha\in\nset_0^n$, see e.g.\ \cite{borel95,polya38,boas39a,bloom53,sherman64,horn77,berg79,kowalski84,duran89,
hoischen92,schmud24signedArxiv,didio24tsystemshomepageArxiv} for more.
Especially, the signed representing measure $\mu$ can be chosen to be absolutely continuous with respect to the Lebesgue measure with a Schwartz function density \cite{duran89}.
It is not possible to combine \cite{duran89} with the non-negativity of the moment sequences, i.e., the Schwartz function or at least $C^r$-density can not be assume to be non-negative because $s$ is a moment sequence.
To see this take the moment sequence generated by the characteristic function of $[0,1]$.
Its (non-negative) measure is unique but its signed measures are not.
In fact, on $\rset^n$ no signed measure is unique.
But we will see how uniqueness comes into play on $\rset^n$ via the characteristic function.

In what follows we remind the reader of the following.
Given a bounded measure $\mu$ on $\rset^n$, i.e., $\mu(\rset^n) < \infty$, then the \emph{characteristic function} $\hat\mu$ of $\mu$ is defined by
\begin{equation}\label{eq:characFct}
\hat\mu(z) := \int_{\rset^n} e^{i\cdot\langle x,z\rangle}~\diff\mu(x)
\end{equation}
for all $z\in\rset^n$.
The characteristic function $\hat\mu$ is uniformly continuous and bounded \cite[Prop.\ 3.8.4 (i)]{bogachevMeasureTheory}.
It is connected to the moments
\[s_\alpha := \int_{\rset^n} x^\alpha~\diff\mu(x)\]
for all $\alpha\in\nset_0^n$ of $\mu$ by
\[s_\alpha = (-i)^{|\alpha|}\cdot (\partial^\alpha \hat\mu)(0),\]
as long as $s_\alpha$ exists, i.e., $\hat\mu$ is at $x=0$ the moment generating function \cite[Prop.\ 2.5 (ix)]{sato99}.
$\hat\mu$ characterizes $\mu$ uniquely \cite[Prop.\ 3.8.6]{bogachevMeasureTheory}.
See e.g.\ \cite[Ch.\ 3.8]{bogachevMeasureTheory} for more.
What is not contained in \cite{bogachevMeasureTheory} can be found e.g.\ in \cite[Prop.\ 2.5 (xii)]{sato99}, \cite[Thm.\ 9.5.4]{dudley89}, or \cite[p.\ 171]{linnik77}:
If $\hat\mu\in L^1(\rset^n,\cset)$ then $\mu$ is absolutely continuous with respect to the Lebesgue measure, i.e., $\diff\mu(x) = g(x)\,\diff x$, with a bounded and continuous density $g$ explicitly given by
\begin{equation}\label{eq:density}
g(x) = (2\pi)^{-n}\cdot\int_{\rset^n} e^{-i\cdot\langle x,z\rangle}\cdot\hat\mu(z)~\diff z
\end{equation}
for all $x\in\rset^n$.
From that we get the following.

\begin{thm}\label{thm:fourierDensity}
Let $n\in\nset$ and let $s=(s_\alpha)_{\alpha\in\nset_0^n}$ be a real sequence such that
\[f(z):=\sum_{\alpha\in\nset_0^n}\frac{i^{|\alpha|}\cdot s_\alpha}{\alpha!}\cdot z^\alpha\] 
is entire and $f\in L^1(\rset^n,\cset)$.
Then the following hold:
\begin{enumerate}[(i)]
\item $s$ is represented by a signed absolutely continuous measure $\mu$ with respect to the Lebesgue measure, i.e., $\diff\mu(x) = g(x)\,\diff x$, and the density $g$ is bounded, continuous, and explicitly given by
\[g(x) = (2\pi)^{-n}\cdot\int_{\rset^n} e^{-i\cdot\langle x,z\rangle}\cdot f(z)~\diff z\]
for all $x\in\rset^n$.

\item If $g\geq 0$ on $\rset^n$ in (i) then $s$ is a determinate moment sequence with representing measure $\mu$ given by $\diff\mu(x) = g(x)\,\diff x$.
\end{enumerate}
\end{thm}
\begin{proof}
(i): Since the Fourier transformations (\ref{eq:characFct}) and (\ref{eq:density}) are linear in $\mu$ resp.\ $g$ we have that \cite[Prop.\ 2.5 (xii)]{sato99} also holds for signed measures $\mu$ with $|\mu|(\rset^n)<\infty$ and signed functions $g$.
Since $f$ is entire and all derivatives at $x=0$ coincide with the $s_\alpha$ we have that $f$ is the characteristic function $\hat\mu$ of a signed measure $\mu$ \cite[Prop.\ 3.8.6]{bogachevMeasureTheory}.
Hence, by \cite[Prop.\ 2.5 (xii)]{sato99} we have that $s$ is represented by a signed measure $\mu$ which is absolutely continuous with respect to the Lebesgue measure and its density $g$ is bounded, continuous and explicitly given by (\ref{eq:density}).

(ii): If $g\geq 0$ then $\mu$ given by $\diff\mu(x):=g(x)\,\diff x$ is non-negative and hence $s$ is a moment sequence with representing measure $\mu$.
Since $f = \hat\mu$ is entire it is uniquely determined by $s$ by \cite[Prop.\ 3.8.6]{bogachevMeasureTheory}.
\end{proof}

\begin{exm}
Let $s = (1,0,1,0,1,0,1,0,\dots)$.
Then
\[f(z) = \sum_{k\in\nset_0} \frac{i^{2k}\cdot 1}{(2k)!}\cdot z^{2k} = e^{-z^2}\]
is entire and $f\in L^1(\rset,\cset)$.
By \Cref{thm:fourierDensity} we have that $s$ is represented by an absolutely continuous representing measure $\mu$ with density
\[g(x) = \frac{1}{2\pi}\int_\rset e^{-ixz}\cdot f(z)~\diff z = \frac{1}{2\sqrt{\pi}}\cdot e^{-\frac{1}{4}x^2},\]
see \cite[Exm.\ 3.8.2]{bogachevMeasureTheory}.
Since $g\geq 0$ on $\rset$ we have that $s$ is a determinate moment sequence with absolutely continuous representing measure $\mu$.
\exmsymbol
\end{exm}

\begin{exm}
Let $s = (1,0,0,0,-1,0,0,0,1,0,0,0,-1,0,0,0,\dots)$.
Then
\[f(z) = \sum_{k\in\nset_0} \frac{i^{4k}\cdot (-1)^k}{(4k)!}\cdot z^{4k} = e^{-z^4}\]
is entire and $f\in L^1(\rset)$.
By \Cref{thm:fourierDensity} we have that $s$ is represented by an absolutely continuous representing measure $\mu$ with density
\[g(x) = \frac{1}{2\pi}\int_\rset e^{-ixz-z^4}~\diff z.\]
Since the Fourier transform maps the Schwartz space $\cS(\rset,\cset)$ onto itself \cite[Cor.\ 2.2.15]{grafak10} and since $f\in\cS(\rset,\cset)$ we have that $g\in\cS(\rset,\cset)$.
But since $s_4 = -1$ we have that $s$ is not a moment sequence, i.e., $g\not\geq 0$ on $\rset$.
\exmsymbol
\end{exm}

\begin{cor}\label{cor:densityC0infty}
Let $n\in\nset$, let $g\in C_c^\infty(\rset^n,\rset)$, and set
\[s_\alpha := \int_{\rset^n} x^\alpha\cdot g(x)~\diff x\]
for all $\alpha\in\nset_0^n$.
Then
\[f(z) := \sum_{\alpha\in\nset_0^n} \frac{i^{|\alpha|}\cdot s_\alpha}{\alpha!}\cdot z^\alpha\]
is entire, $f\in L^1(\rset^n,\cset)$, and
\[g(x) = (2\pi)^{-n}\cdot \int_{\rset^n} e^{-i\cdot\langle x,z\rangle}\cdot f(z)~\diff z\]
holds for all $x\in\rset^n$.
\end{cor}

\begin{rem}
From a numerical point of view the densities $g$ in \Cref{thm:fourierDensity} and \Cref{cor:densityC0infty} can of course be approximated since
\begin{equation}\label{eq:gapprox}
g(x) = \lim_{R\to\infty} \lim_{d\to\infty} \int_{[-R,R]^n} e^{-i\cdot\langle x,z\rangle}\cdot \sum_{\alpha\in\nset_0^n:|\alpha|\leq d} \frac{i^{|\alpha|}\cdot s_\alpha}{\alpha!}\cdot z^\alpha~\diff z
\end{equation}
holds for all $x\in\rset^n$.
The explicit calculation of $g$ also allows to evaluate resp.\ determine $g$ only at certain points $x_0\in\rset^n$ or in regions $K\subseteq\rset^n$.
\exmsymbol
\end{rem}

We see that the characteristic function $\hat\mu$ and the density $g$ are connected via the Fourier transform.
The conditions on $\hat\mu$ in \Cref{thm:fourierDensity} and \Cref{cor:densityC0infty} can considerably be weakened to spaces where the Fourier transform and hence its inverse are still defined, e.g., $L^1(\rset^n,\cset)\cap L^2(\rset^n,\cset)$ or $L^2(\rset^n,\cset)$ instead of $C_c^\infty(\rset^n,\cset)$ and (\ref{eq:density}) holds almost everywhere, see e.g.\ \cite{grafak10} and \Cref{rem:invFourierMeasure}.
But then for $L^2(\rset^n,\cset)$ for example an explicit formula like (\ref{eq:density}) for the density does not hold in general since the Fourier transform is gained by continuity due to the Plancherel identity from the Schwartz space $\cS(\rset^n,\cset)$ to $L^2(\rset^n,\cset)$.
However, on $L^2(\rset^n,\cset)$ the approximation (\ref{eq:gapprox}) holds for almost every $x\in\rset^n$.

What happens when $f\notin L^1(\rset^n,\cset)$ is demonstrated in the following two examples with $f\notin L^2(\rset^n,\cset)$ and $f\in L^2(\rset^n,\cset)$, respectively.

\begin{exm}[$f\notin L^1(\rset,\cset)$ and $f\notin L^2(\rset,\cset)$]
Let $s=(1,1,1,\dots)$.
Then
\[f(z) = \sum_{k\in\nset_0} \frac{i^k\cdot 1}{k!}\cdot z^k = e^{iz}.\]
We have $f\notin L^1(\rset,\cset)$ and $f\notin L^2(\rset,\cset)$.
We will see that $s$ is not represented by an absolutely continuous representing measure.
In fact, since a measure $\mu$ is uniquely determined by its characteristic function $\hat\mu$ \cite[Prop.\ 3.8.6]{bogachevMeasureTheory} from
\[\widehat{\delta_1}(z) = \int_\rset e^{ixz}~\diff\delta_1(x) = e^{iz} = f(z)\]
we see that $s$ is the determinate moment sequence represented by $\mu = \delta_1$.
\exmsymbol
\end{exm}

\begin{exm}[$f\notin L^1(\rset,\cset)$ but $f\in L^2(\rset,\cset)$]
Let $s = ((k+1)^{-1})_{k\in\nset_0}$.
Then
\[f(z)=\sum_{k\in\nset_0}\frac{i^k}{k!\cdot (k+1)}\cdot z^k=\frac{e^{iz}-1}{iz}.\]
We have $f\notin L^1(\rset,\cset)$ but since $f$ is continuous at $z=0$ we have at least $f\in L^2(\rset,\cset)$.
With $\diff\mu = \chi_{[0,1]}(x)\,\diff x$ where $\chi_{[0,1]}$ is the characteristic function of the set $[0,1]$ we see
\[\hat\mu(z) = \int_\rset e^{ixz}\cdot\chi_{[0,1]}(x)~\diff x = \frac{e^{iz}-1}{iz} = f(z)\]
for all $z\in\rset$.
Since a measure $\mu$ is uniquely determined by its characteristic function $\hat\mu$ \cite[Prop.\ 3.8.6]{bogachevMeasureTheory} we have that $\mu$ is absolutely continuous with respect to the Lebesgue measure with density $g = \chi_{[0,1]}$.
Since we only have $f\in L^2(\rset,\cset)\cap C(\rset,\cset)$ this makes the calculation of $g$ with (\ref{eq:density}) slightly more cumbersome since additional arguments in which sense the integral exists are necessary.
\exmsymbol
\end{exm}

We also allow now complex sequences $s = (s_\alpha)_{\alpha\in\nset_0^n}$ since this leads by linearity to complex absolutely continuous measures.
The inverse Fourier transform on $L^2(\rset^n,\cset)$ is denoted by $\check{f}$.
Note, for measures and functions different $2\pi$-factors are used.
Hence, we get an additional factor in the calculation of $g$ to convert between these two since $\check{f}$ is the inverse Fourier transform for a function.
The Fourier transform is bijective on $L^2(\rset^n,\cset)$.
We get the following.

\begin{thm}\label{thm:fourierDensityL2}
Let $n\in\nset$ and let $s=(s_\alpha)_{\alpha\in\nset_0^n}$ be a real (or complex) sequence such that
\[f(z):=\sum_{\alpha\in\nset_0^n} \frac{i^{|\alpha|}\cdot s_\alpha}{\alpha!}\cdot z^\alpha\]
is entire and $f\in L^2(\rset^n,\cset)$.
Then the following hold:
\begin{enumerate}[(i)]
\item $s$ is represented by a complex absolutely continuous representing measure $\mu$ with respect to the Lebesgue measure, i.e., $\diff\mu(x) = g(x)\,\diff x$, and the density $g$ is given by the inverse Fourier transform
\[g = (2\pi)^{-n/2}\cdot \check f \quad\in L^2(\rset^n,\cset).\]

\item If $g\geq 0$ almost everywhere on $\rset^n$ in (i) then $s$ is a determinate moment sequence with representing measure $\mu$ given by $\diff\mu(x) = g(x)\,\diff x$.
\end{enumerate}
\end{thm}
\begin{proof}
Similar to the proof of \Cref{thm:fourierDensity}.
Use the Fourier transform on $L^2(\rset^n,\cset)$ \cite{grafak10} and the fact that a measure $\mu$ is uniquely determined by its characteristic function $f = \hat\mu$ \cite[Prop.\ 3.8.6]{bogachevMeasureTheory}.
\end{proof}

If $g\in L^2(\rset^n,\cset)$ is compactly supported we even have the reverse direction in \Cref{thm:fourierDensityL2}:
\[g\in L^2(\rset^n,\cset)\ \text{has compact support} \quad\Rightarrow\quad f\in L^2(\rset^n,\cset)\ \text{is entire}.\]
Compact support of a representing measure of a moment sequence can be checked by the growth of the moments, see e.g.\ \cite[Prop.\ 4.1]{schmudMomentBook}.
We get the following.

\begin{cor}\label{cor:densityL2Compact}
Let $n\in\nset$ and let $s = (s_\alpha)_{\alpha\in\nset_0^n}$ be a real sequence.
Then the following are equivalent:
\begin{enumerate}[(i)]
\item The function
\[f(z) = \sum_{\alpha\in\nset_0^n} \frac{i^{|\alpha|}\cdot s_\alpha}{\alpha!}\cdot z^\alpha\]
is entire and $f\in L^2(\rset^n,\cset)$ such that $g := (2\pi)^{-n/2}\cdot \check{f}$ fulfills
\begin{enumerate}[(a)]
\item $g\in L^2(\rset^n,[0,\infty))$, i.e., $g\geq 0$ on $\rset^n$ almost everywhere, and
\item $\supp g \subseteq [-c,c]^n$ for some $c\geq 0$.
\end{enumerate}

\item The function
\[f(z) = \sum_{\alpha\in\nset_0^n} \frac{i^{|\alpha|}\cdot s_\alpha}{\alpha!}\cdot z^\alpha\]
is entire and $f\in L^2(\rset^n,\cset)$ such that $g := (2\pi)^{-n/2}\cdot \check{f}$ is non-negative almost everywhere on $\rset^n$ and there exists a constant $c\geq 0$ such that
\[\lim_{k\to\infty} \sqrt[2k]{s_{2k\cdot e_j}} \leq c\]
holds for all $j=1,\dots,n$.

\item $s$ is a determinate moment sequence with an absolutely continuous representing measure $\mu$, i.e., $\diff\mu(x) = g(x)\,\diff x$, with density $g\in L^2(\rset^n,[0,\infty))$ such that $\supp g \subseteq [-c,c]^n$ for some constant $c\geq 0$.
\end{enumerate}
\end{cor}

Regarding the differentiability of the $L^2$-density $g$ with compact support we can combine the results of \Cref{sec:full01} with \Cref{cor:densityL2Compact} to get the following immediate result.

\begin{cor}
Let $n\in\nset$, let $d\in\nset_0$, and let $s = (s_\alpha)_{\alpha\in\nset_0^n}$ be a real sequence.
Then the following are equivalent:
\begin{enumerate}[(i)]
\item The function
\[f(z) := \sum_{\alpha\in\nset_0^n} \frac{i^{|\alpha|}\cdot s_\alpha}{\alpha!}\cdot z^\alpha\]
is entire such that
\begin{enumerate}
\item $x^\beta\cdot f\in L^2(\rset^n,\cset)$ for all $\beta\in\nset_0^n$ with $|\beta|\leq d$,
\item $g := (2\pi)^{-n/2}\cdot \check{f} \geq 0$ almost everywhere on $\rset^n$, and
\item there exists a constant $c\geq 0$ such that $\supp g\subseteq [-c,c]^n$.
\end{enumerate}

\item The function
\[f(z):=\sum_{\alpha\in\nset_0^n} \frac{i^{|\alpha|}\cdot s_\alpha}{\alpha!}\cdot z^\alpha\]
is entire such that
\begin{enumerate}
\item $x^\beta\cdot f\in L^2(\rset^n,\cset)$ for all $\beta\in\nset_0^n$ with $|\beta|\leq d$,
\item $g := (2\pi)^{-n/2}\cdot \check{f} \geq 0$ almost everywhere on $\rset^n$, and
\item there exists a constant $c\geq 0$ such that
\[\lim_{k\to\infty} \sqrt[2k]{s_{2k\cdot e_j}} \leq c\]
holds for all $j=1,\dots,n$.
\end{enumerate}

\item For all $\beta\in\nset_0^n$ with $|\beta|\leq d$ the functions
\[f_\beta(z) := \sum_{\alpha\in\nset_0^n} \frac{i^{|\alpha|}\cdot (\partial^\beta s_\alpha)}{\alpha!}\cdot z^\alpha\]
are entire such that
\begin{enumerate}[(a)]
\item $f_\beta\in L^2(\rset^n,\cset)$ for all $\beta\in\nset_0^n$ with $|\beta|\leq d$,
\item $g := (2\pi)^{-n/2}\cdot \check{f_0} \geq 0$ almost everywhere on $\rset^n$, and
\item there exists a constant $c\geq 0$ such that $\supp g\subseteq [-c,c]^n$.
\end{enumerate}

\item For all $\beta\in\nset_0^n$ with $|\beta|\leq d$ the functions
\[f_\beta(z) := \sum_{\alpha\in\nset_0^n} \frac{i^{|\alpha|}\cdot (\partial^\beta s_\alpha)}{\alpha!}\cdot z^\alpha\]
are entire such that
\begin{enumerate}[(a)]
\item $f_\beta\in L^2(\rset^n,\cset)$ for all $\beta\in\nset_0^n$ with $|\beta|\leq d$,
\item $g := (2\pi)^{-n/2}\cdot \check{f_0} \geq 0$ almost everywhere on $\rset^n$, and
\item there exists a constant $c\geq 0$ such that
\[\lim_{k\to\infty} \sqrt[2k]{s_{2k\cdot e_j}} \leq c\]
holds for all $j=1,\dots,n$.
\end{enumerate}

\item $s$ is a determinate moment sequence with an absolutely continuous representing measure $\mu$, i.e., $\diff\mu(x) = g(x)\,\diff x$, with a density $g\in L^2(\rset^n,\cset)$ such that
\begin{enumerate}[(a)]
\item $\partial^\beta g \in L^2(\rset^n,\cset)$ for all $\beta\in\nset_0^n$ with $|\beta|\leq d$,
\item $g\geq 0$ almost everywhere on $\rset^n$, and
\item $\supp g \subseteq [-c,c]^n$.
\end{enumerate}
\end{enumerate}
\end{cor}
\begin{proof}
Follows from \Cref{cor:densityL2Compact} with the fact that $\partial^\beta$-differentiation under the Fourier transform becomes $x^\beta$-multiplication, see e.g.\ \cite[Prop.\ 2.2.11]{grafak10}.
The $f_\beta$ are the characteristic functions $f_\beta = \widehat{\partial^\beta\mu}$ of the (distributional) derivatives $\partial^\beta\mu$ of the measure $\mu$.
All $\partial^\beta\mu$ are unique since $\diff(\partial^\beta\mu)(x) = (\partial^\beta g)(x)\,\diff x$ and hence $\supp\partial^\beta\mu\subseteq [-c,c]^n$.
\end{proof}

\begin{rem}\label{rem:invFourierMeasure}
We have seen that if $\hat\mu$ is $L^1$ or $L^2$ then we get a unique absolutely continuous representing measure $\mu$.
Additionally, the characteristic function $\hat\mu$ determines the measure $\mu$ uniquely \cite[Prop.\ 3.8.6]{bogachevMeasureTheory}.
The reader might therefore wonder if we can apply the (inverse) Fourier transformation to $\hat\mu$ even if $\mu$ is not absolutely continuous with an $L^1$- or $L^2$-density.
This is in fact possible and known as \emph{L\'evy's Inversion Formula} (\ref{eq:levyInversion}), see e.g.\ \cite[Sec.\ 16.6]{williams91probMartingales} or \cite[Thm.\ 26.2]{billingsley95probMeasure} for the $1$-dimensional version:
\medskip

\textit{
Let $-\infty < a < b < \infty$ and let $\mu$ be a probability measure.
% with $\mu(\{a\}) = \mu(\{b\}) = 0$.
Then}
\begin{equation}\label{eq:levyInversion}
\frac{\mu(\{a\})}{2} + \mu((a,b)) + \frac{\mu(\{b\})}{2} \;=\; \lim_{T\nearrow\infty} \frac{1}{2\pi} \int_{-T}^T \frac{e^{-ita} - e^{-itb}}{it}\cdot \hat{\mu}(t)~\diff t.
\end{equation}\smallskip

\noindent
Hence, if $\hat\mu = f$ is entire and therefore given by the moments $s_k$ for all $k\in\nset_0$ we have that
\[\frac{\mu(\{a\})}{2} + \mu((a,b)) + \frac{\mu(\{b\})}{2} \;=\; \lim_{T\nearrow\infty} \frac{1}{2\pi} \int_{-T}^T\!\! \frac{e^{-ita} - e^{-itb}}{it}\cdot \left(\sum_{k\in\nset_0} \frac{i^k\cdot s_k}{k!}\cdot t^k\right)\diff t\]
holds for all $a<b$.
On $\rset^n$ with $n\in\nset$ a similar but more cumbersome formula like (\ref{eq:levyInversion}) exists \cite{folland92fourierAnalysis}.
But with the Fourier and inverse Fourier transformation 
\[(\cF \varphi)(x) := \int_{\rset^n} \varphi(z)\cdot e^{-i\cdot x\cdot z}~\diff z
\;\;\text{and}\;\;
(\cF^{-1} \varphi)(x) := \frac{1}{(2\pi)^n}\int_{\rset^n} \varphi(z)\cdot e^{i\cdot x\cdot z}~\diff z\]
for $\varphi\in\cS(\rset^n,\cset)$ and the Fourier transformation of temperate distributions $\cS(\rset^n,\cset)'$, like finite measures $\mu$, see e.g.\ \cite[Sec.\ 5.3]{grubbDistributions}, i.e., we have
\[\mu(\varphi) = \int_{\rset^n} \varphi(x)~\diff\mu(x) = \langle \varphi,\mu\rangle = \langle\cF\cF^{-1}\varphi,\mu\rangle = \langle\cF^{-1}\varphi,\cF\mu\rangle = \langle\cF^{-1}\varphi,\hat\mu\rangle,\]
we get a simple and explicit formula for the measure $\mu$ by
\begin{equation}\label{eq:inverseMeasure}
\mu(\varphi) = \frac{1}{(2\pi)^n} \int_{\rset^n} \left(\int_{\rset^n} \varphi(x)\cdot e^{i\cdot x\cdot z}~\diff x\right) \cdot \left(\sum_{\alpha\in\nset_0^n} \frac{i^{|\alpha|}\cdot s_\alpha}{\alpha!}\cdot z^\alpha\right) ~\diff z
\end{equation}
for all $\varphi\in\cS(\rset^n,\cset)$ depending on the entire characteristic function $\hat\mu = f$, i.e., depending only on the moments $s_\alpha = \mu(x^\alpha)$ of the measure $\mu$.

Note that the characteristic function can be extended from $\rset^n$ to a general Banach space, see e.g.\ \cite[Sec.\ E.1.c]{hytonen17AnalysisBanachSpaces2}.
Hence, more general treatments then we did here in the current work on $\rset^n$ with $\rset[x_1,\dots,x_n]$ are possible.
\exmsymbol
\end{rem}

We want to end this section with the following result characterizing some multi-variate moment sequences by the positive semi-definiteness of certain Hermitian $d\times d$-matrices.
For $n\geq 2$ checking the positive semi-definiteness of the Hankel matrices $(s_{\alpha+\beta})_{\alpha,\beta\in\nset_0^n: |\alpha|,|\beta|\leq d}$ is not sufficient \cite{schmud79}.
We use \emph{Bochner's Theorem} \cite[p.\ 76, Satz 23]{bochner32} (or see e.g.\ \cite[Prop.\ 2.5 (i)]{sato99}, \cite[Thm.\ 1.23]{duoandi01}, or \cite[p.\ 273, Thm.\ 21]{widder72}) and we call the following therefore the \emph{Bochner Test}.

\begin{thm}[Bochner Test]\label{thm:bochnerTest}
Let $n\in\nset$ and $s = (s_\alpha)_{\alpha\in\nset_0^n}$ be a real sequence such that $s_0 = 1$ and
\[f(z):=\sum_{\alpha\in\nset_0^n}\frac{i^{|\alpha|}\cdot s_\alpha}{\alpha!}\cdot z^\alpha\]
is entire.
Set
\[f_1(z) := \sum_{\substack{\alpha\in\nset_0^n:\\ |\alpha|\in 2\nset_0+1}} \frac{i^{|\alpha|}\cdot s_\alpha}{\alpha!}\cdot z^\alpha \quad\text{and}\quad f_2(z) := \sum_{\substack{\alpha\in\nset_0^n:\\ |\alpha|\in 2\nset_0}} \frac{i^{|\alpha|}\cdot s_\alpha}{\alpha!}\cdot z^\alpha.\]
Then the following are equivalent:
\begin{enumerate}[(i)]
\item $s$ is a determinate moment sequence.

\item $\big(f(z_j-z_k)\big)_{j,k=1}^d \succeq 0$ holds for all $d\in\nset$ and $z_1,\dots,z_d\in\rset^n$.

\item For all $d\in\nset$ and $z_1,\dots,z_d\in\rset^n$ we have
\[\big(f_2(z_j-z_k)\big)_{j,k=0}^d \succeq 0 \quad\text{and}\quad \big(f_2(z_j-z_k)\big)_{j,k=0}^d \succeq \big(f_1(z_j-z_k)\big)_{j,k=0}^d.\]
\end{enumerate}
\end{thm}
\begin{proof}
(i) $\Leftrightarrow$ (ii):
By Bochner's Theorem a continuous function $f$ with $f(0)=1$ is the characteristic function of a bounded measure if and only if (ii) holds.
Determinacy of $s$ follows again from $f$ being entire and hence unique.
Note, $(f(z_j-z_k))_{j,k=1}^d \succeq 0$ is well-defined since $(f(z_j-z_k))_{j,k=1}^d$ is Hermitian.

(ii) $\Leftrightarrow$ (iii): 
We have $f = f_1 + f_2$ with $f_1$ odd and $f_2$ even, $(f_1(z_j-z_k))_{j,k=1}^d$ Hermitian, and $(f_2(z_j-z_k))_{j,k=1}^d$ real symmetric.
Hence, the equivalence follows from replacing $z_1,\dots,z_d\in\rset^n$ by $-z_1,\dots,-z_d\in\rset^n$.
\end{proof}

To determine if $s$ is a moment sequence by Haviland's Theorem \cite{havila35,havila36} one has to check $L_s(p)\geq 0$ for all $p\in\rset[x_1,\dots,x_n]$ with $p\geq 0$ where $L_s$ is the Riesz functional, i.e., $L_s(x^\alpha) = s_\alpha$ and $L_s$ is linear on $\rset[x_1,\dots,x_n]$.
This is NP-hard.
Only in the case $n=1$ this reduces to checking $(s_{j+k})_{j,k=0}^d\succeq 0$ for all $d\in\nset_0$ since here all non-negative polynomials are sums of squares \cite{hilbert88}.

Since testing sequences for being a moment sequence is NP-hard any simplification can only hold for some sequences.
In \Cref{thm:bochnerTest} for sequences $s$ such that $f$ is entire this reduces to checking complex positive semi-definiteness of the complex $d\times d$-matrices $(f(z_j-z_k))_{j,k=1}^d$.
Sequences $s$ with
\[\lim_{k\to\infty} \sqrt[2k]{s_{2k\cdot e_j}} \leq c\]
for all $j=1,\dots,n$ and for some $c\geq 0$ are such sequences which can be checked by \Cref{thm:bochnerTest}.
If $f$ is not entire then $s_\alpha$ with $\alpha\neq 2k\cdot e_j$ are clearly violating coming from a measure supported on $[-c,c]^n$.
\Cref{thm:bochnerTest} covers entire $f$'s.

\section{Summary}%%%
%%%%%%%%%%%%%%%%%%%%
\label{sec:summary}

In this work we dealt with three cases of problems finding and characterizing absolutely continuous representing measures for (signed) moment sequences resp.\ functionals:
\begin{enumerate}[\;(a)]
\item truncated moment functionals (\Cref{sec:truncated}),

\item full (signed) moment functionals on $[0,1]^n$ (\Cref{sec:full01}), and

\item full (signed) moment functionals on $\rset^n$ (\Cref{sec:fullR}).
\end{enumerate}
Each of these problems is attacked by its own method:
\begin{enumerate}[(a')]
\item Dirac approximating sequences (\Cref{dfn:dirac}),

\item the \Cref{thm:signedHaus}, and

\item characteristic functions and Fourier transform \cite[Prop.\ 2.5 (xii)]{sato99}.
\end{enumerate}

In \Cref{sec:truncated} we generalize the work \cite{ambros14}.
We use the concept of a Dirac-approximating family to show that \cite{ambros14} can easily be extended, i.e., that as long there exists a Dirac-approximating family every moment sequence resp.\ moment functional in the moment cone is representable by an absolutely continuous representing measure.
In \Cref{exm:discontinuous} we give an example of an interior moment sequence resp.\ moment functional that can not be represented by an absolutely continuous representing measure since in this example there exists no Dirac-approximating family.
We continue with discussions and corollaries to show that the density functions can be chosen to be $L^1$-, $C^\infty$-, or even $C_c^\infty$-functions.
We also show that representing measures absolutely continuous with respect to more than the Lebesgue measure are also implied by \Cref{thm:main}.
\Cref{thm:main} also includes the cases of more general measurable sets $\cX$.
Also boundary moment sequence can be represented by absolutely continuous representing measures.
The proof of \Cref{thm:main} is constructive, i.e., it provides a way to gain absolutely representing measures.

In \Cref{sec:full01} we considered sequences $s$ resp.\ linear functionals $L_s:\rset[x_1,\dots,x_n]\to\rset$ which can be represented by signed absolutely continuous measures supported on $[0,1]^n$.
We give full characterizations of such sequences especially when additional regularity on the density is imposed.
The main technique is the \Cref{thm:signedHaus} and derivatives of (moment) sequences \cite{didio23gaussian}.

Finally, in \Cref{sec:fullR} we consider full moment functionals on $\rset^n$.
Because of the great flexibility on $\rset^n$, especially with signed measures, we have to make the restriction that the (signed) moment sequence $s = (s_\alpha)_{\alpha\in\nset_0^n}$ is such that
\[f(z) := \sum_{\alpha\in\nset_0^n} \frac{i^{|\alpha|}\cdot s_\alpha}{\alpha!}\cdot z^\alpha\]
is entire.
Of course, $f$ is then the characteristic function $\hat\mu$ of a unique (signed) measure $\mu$.
The function $f$ is also known as moment generating function.
As in signal processing the image (here the measure) is recovered by the inverse Fourier transform of $f$.
Since the regularity of the density of an absolutely continuous measure can be controlled on the Fourier side we find full characterizations of such densities.

It shall be noted that while the (inverse) Fourier approach is well-known and widely applied we did not find references from the moment theoretic side dealing with this tool of the moment generating function resp.\ characteristic function $f$.
In the standard literature \cite{achieser56,ahiezer62,akhiezClassical,kreinMarkovMomentProblem,marshallPosPoly,lauren09} about moments we did not find explicit treatments.
Even in the seminal work \cite{schmudMomentBook} the moment generating function appears only in exercise 8 on page 91 but no deeper treatment is provided.
It is therefore interesting which properties and characterizations of moment sequences, moment functionals, and their representing measures lie ahead, especially when the full power of Fourier analysis is employed, see e.g.\ \cite{duoandi01,stein03,grafak10,grafak09,bahour11} and \Cref{rem:invFourierMeasure}.

Already characterizing moment sequences, moment functionals, and their representing measures is (NP) hard.
A reconstruction or at least approximation of representing measures goes even a step further.
It is therefore surprising that for the cases considered in \Cref{sec:fullR} the simple and explicit formulas (\ref{eq:density}), (\ref{eq:levyInversion}), and (\ref{eq:inverseMeasure}) hold.
They are based on the inverse Fourier transform of the characteristic function $\hat\mu=f$ and provide the representing measures $\mu$, see \Cref{thm:fourierDensity}, \Cref{thm:fourierDensityL2}, and \Cref{rem:invFourierMeasure}.

\section*{Funding}%%%
%%%%%%%%%%%%%%%%%%%%%

The author and this project are supported by the Deutsche Forschungs\-gemein\-schaft DFG with the grant DI-2780/2-1 and his research fellowship at the Zukunfts\-kolleg of the University of Konstanz, funded as part of the Excellence Strategy of the German Federal and State Government.

%%\bibliographystyle{amsplain}
%\bibliographystyle{amsalpha}
%\bibliography{../../bibdata}

\begin{thebibliography}{HvNVW17}

\bibitem[Ach56]{achieser56}
N.~I. Achieser, \emph{Theory of {A}pproximation}, Frederick Ungar, New York,
  1956.

\bibitem[AK62]{ahiezer62}
N.~I. Ahiezer and M.~Kre\u{\i}n, \emph{Some {Q}uestions in the {T}heory of
  {M}oments}, American Mathematical Society, Providence, Rhode Island, 1962.

\bibitem[Akh65]{akhiezClassical}
N.~I. Akhiezer, \emph{The classical moment problem and some related questions
  in analysis}, Oliver \& Boyd, Edinburgh, 1965.

\bibitem[Amb14]{ambros14}
C.-G. Ambrozie, \emph{A {R}iesz--{H}aviland {T}ype {R}esult for {T}runcated
  {M}oment {P}roblems with {S}olutions in ${L}^1$}, J.~Op.\ Theory \textbf{71}
  (2014), 85--93.

\bibitem[Ana06]{anast06}
G.~A. Anastassiou, \emph{Applications of geometric moment theory related to
  optimal portfolio management}, Comput.\ Math.\ Appl. \textbf{51} (2006),
  1405--1430.

\bibitem[APST19]{ammari19}
H.~Ammari, M.~Putinar, A.~Streenkamp, and F.~Triki, \emph{Identification of an
  algebraic domain in two dimensions from a finite number of its generalized
  polarization tensors}, Math.\ Ann. \textbf{375} (2019), 1337--1354.

\bibitem[ARS18]{amendo18}
C.~Am\'{e}ndola, K.~Ranestad, and B.~Sturmfels, \emph{Algebraic
  {I}dentifiability of {G}aussian {M}ixtures}, Int.\ Math.\ Res.\ Notices
  \textbf{2018} (2018), no.~21, 6556--6580.

\bibitem[Bal61]{balins61}
M.~L. Balinski, \emph{An algorithm for finding all vertices of convex
  polyhedral sets}, J.~Soc.\ Indust.\ Appl.\ Math. \textbf{9} (1961), 72--88.

\bibitem[BCD11]{bahour11}
H.~Behouri, J.-Y. Chemin, and R.~Danchin, \emph{Fourier {A}nalysis and
  {N}onlinear {P}artial {D}ifferential {E}quations}, Springer-Verlag, Berlin,
  Heidelberg, 2011.

\bibitem[BCJ79]{berg79}
C.~Berg, J.~P.~R. Christensen, and C.~U. Jensen, \emph{A remark on the
  multidimensional moment problem}, Math.\ Ann. \textbf{243} (1979), 163--169.

\bibitem[BF20]{blekhe20}
G.~Blekherman and L.~Fialkow, \emph{The core variety and representing measures
  in the truncated moment problem}, J.~Op.\ Theory \textbf{84} (2020),
  185--209.

\bibitem[BGL07]{becker07}
B.~Beckermann, G.~H. Golub, and G.~Labahn, \emph{On the numerical condition of
  a generalized {H}ankel eigenvalue problem}, Numer.\ Math. \textbf{106}
  (2007), 41--68.

\bibitem[Bil95]{billingsley95probMeasure}
P.~Billingsley, \emph{Probability and {M}easure}, John Wiley \& Sons, New York,
  1995.

\bibitem[BJL19]{brehard19}
F.~Br\'ehard, M.~Joldes, and J.-B. Lasserre, \emph{On a moment problem with
  holonomic functions}, {ISSAC} '19: {P}roceedings of the 2019 on
  {I}nternational {S}ymposium on {S}ymbolic and {A}lgebraic {C}omputation,
  2019, pp.~66--73.

\bibitem[Blo53]{bloom53}
M.~Bloom, \emph{On the total variation of solutions of the bounded variation
  moment problem}, Proc.\ Amer.\ Math.\ Soc. \textbf{4} (1953), 118--126.

\bibitem[Boa39]{boas39a}
R.~P. Boas, \emph{The {S}tieltjes moment problem for functions of bounded
  variation}, Bull.\ Amer.\ Math.\ Soc. \textbf{45} (1939), 399--404.

\bibitem[Boc32]{bochner32}
S.~Bochner, \emph{Vorlesung \"{u}ber {F}ouriersche {I}ntegrale}, Akademie
  Verlag, Leipzig, 1932, reprinted unchanged in Chelsea Publishing Company, New
  York (1948).

\bibitem[Bog07]{bogachevMeasureTheory}
V.~I. Bogachev, \emph{Measure {T}heory}, Springer-Verlag, Berlin, 2007.

\bibitem[Bor95]{borel95}
E.~Borel, \emph{Sur quelques points de la th\'eorie des fonctions}, Ann.\ Ecole
  Norm.\ Sup. \textbf{12} (1895), 9--55.

\bibitem[BT91]{berg91}
C.~Berg and M.~Thill, \emph{Rotation invariant moment problems}, Acta Math.
  \textbf{167} (1991), 207--227.

\bibitem[Che93]{chen93}
C.-C. Chen, \emph{Improved moment invariants for shape discrimination}, Pattern
  Recognit. \textbf{26} (1993), 683--686.

\bibitem[DBN92]{dai92}
M.~Dai, P.~Baylou, and M.~Najim, \emph{An efficient algorithm for computation
  of shape moments from run-length codes or chain codes}, Pattern Recognit.
  \textbf{25} (1992), 1119--1128.

\bibitem[dD19]{didio18gaussian}
P.~J. di~Dio, \emph{The multidimensional truncated {M}oment {P}roblem: Gaussian
  and {L}og-{N}ormal {M}ixtures, their {C}arath\'eodory {N}umbers, and {S}et of
  {A}toms}, Proc.\ Amer.\ Math.\ Soc. \textbf{147} (2019), 3021--3038.

\bibitem[dD23]{didio23gaussian}
\bysame, \emph{The multidimensional truncated moment problem: {G}aussian
  mixture reconstruction from derivatives of moments}, J.~Math.\ Anal.\ Appl.
  \textbf{517} (2023), 126592.

\bibitem[dD24]{didio24tsystemshomepageArxiv}
\bysame, \emph{An {I}ntroduction to {T}-{S}ystems -- with a special {E}mphasis
  on {S}parse {M}oment {P}roblems, {S}parse {P}ositivstellens{\"a}tze, and
  {S}parse {N}ichtnegativstellens{\"a}tze}, arXiv:2403.04548, Ch.\ 3.4 ``Signed
  Representing Measures: Boas' Theorem'' of the version of April 4, 2024
  available at the authors homepage:
  \href{https://www.uni-konstanz.de/zukunftskolleg/community/philipp-di-dio/}{https://www.uni-konstanz.de/zukunftskolleg/community/philipp-di-dio/}.

\bibitem[dDK21]{didio21HilbertFunction}
P.~J. di~Dio and M.~Kummer, \emph{The multidimensional truncated {M}oment
  {P}roblem: Carath\'eodory {N}umbers from {H}ilbert {F}unctions}, Math.\ Ann.
  \textbf{380} (2021), 267--291.

\bibitem[dDS18a]{didio17w+v+}
P.~J. di~Dio and K.~Schm{\"u}dgen, \emph{The multidimensional truncated
  {M}oment {P}roblem: {A}toms, {D}eterminacy, and {C}ore {V}ariety}, J.~Funct.\
  Anal. \textbf{274} (2018), 3124--3148.

\bibitem[dDS18b]{didio17Cara}
\bysame, \emph{{The} multidimensional truncated {M}oment {P}roblem:
  {C}arath\'eodory {N}umbers}, J.~Math.\ Anal.\ Appl. \textbf{461} (2018),
  1606--1638.

\bibitem[dDS22]{didioCone22}
\bysame, \emph{The multidimensional truncated moment problem: The moment cone},
  J.~Math.\ Anal.\ Appl. \textbf{511} (2022), 126066, 38 pages.

\bibitem[Dud89]{dudley89}
R.~M. Dudley, \emph{Real {A}nalysis and {P}robability}, Chapman \& Hall, New
  York, 1989.

\bibitem[Duo01]{duoandi01}
J.~Duoandikoetxea, \emph{Fourier {A}nalysis}, American Mathematical Society,
  Providence, Rhode Island, 2001, translation of the {S}panish original
  (Addison-Wesley, 1995).

\bibitem[Dur89]{duran89}
A.~J. Duran, \emph{The {S}tieltjes {M}oments {P}roblem for {R}apidly
  {D}ecreasing {F}unctions}, Proc.\ Amer.\ Math.\ Soc. \textbf{107} (1989),
  731--741.

\bibitem[Fia16]{fialkoMomProbSurv}
L.~A. Fialkow, \emph{The truncated {$K$}-moment problem: a survey}, Theta Ser.\
  Adv.\ Math. \textbf{18} (2016), 25--51.

\bibitem[Fol92]{folland92fourierAnalysis}
G.~B. Folland, \emph{Fourier {A}nalysis and its {A}pplications}, Wadsworth \&
  Brooks/Cole, Belmont, California, 1992.

\bibitem[GLPR12]{gravin12}
N.~Gravin, J.~Lasserre, D.~V. Pasechnik, and S.~Robins, \emph{The {I}nverse
  {M}oment {P}roblem for {C}onvex {P}olytopes}, Discrete Comput.\ Geom.
  \textbf{48} (2012), 596--621.

\bibitem[GMV99]{golub99}
G.~H. Golub, P.~Milfar, and J.~Varah, \emph{A stable numberical method for
  inverting shape from moments}, SIAM J.\ Sci.\ Comput. \textbf{21} (1999),
  no.~4, 1222--1243.

\bibitem[GNPR14]{gravin14}
N.~Gravin, D.~Nguyen, D.~V. Pasechnik, and S.~Robins, \emph{The inverse moment
  problem for convex polytopes: Implementation aspects}, arXiv:1409.3130v2.

\bibitem[GPSS18]{gravin18}
N.~Gravin, D.~Pasechnik, B.~Shapiro, and M.~Shapiro, \emph{On moments of a
  polytope}, Anal.\ Math.\ Phys. \textbf{8} (2018), 255--287.

\bibitem[Gra09]{grafak09}
L.~Grafakos, \emph{Modern {F}ourier {A}nalysis}, 2nd ed., Springer, New York,
  2009.

\bibitem[Gra10]{grafak10}
\bysame, \emph{Classical {F}ourier {A}nalysis}, 2nd ed., Springer, New York,
  2010.

\bibitem[Gru09]{grubbDistributions}
G.~Grubb, \emph{Distributions and {O}perators}, Spinger, New York, 2009.

\bibitem[Ham20]{hamburger20}
H.~L. Hamburger, \emph{{\"U}ber eine {E}rweiterung des {S}tieltjesschen
  {M}omentenproblems}, Math.\ Ann. \textbf{81} (1920), 235--319.

\bibitem[Hau21a]{hausdo21}
F.~Hausdorff, \emph{Summationsmethoden und {M}omentenfolgen {I}}, Math.~Z.
  \textbf{9} (1921), 74--109.

\bibitem[Hau21b]{hausdo21a}
\bysame, \emph{Summationsmethoden und {M}omentenfolgen {II}}, Math.~Z.
  \textbf{9} (1921), 280--299.

\bibitem[Hau23]{hausdo23}
\bysame, \emph{Momentprobleme f\"{u}r ein endliches {I}ntervall}, Math.~Z.
  \textbf{16} (1923), 220--248.

\bibitem[Hav35]{havila35}
E.~K. Haviland, \emph{On the momentum problem for distribution functions in
  more than one dimension}, Amer.\ J.\ Math. \textbf{57} (1935), 562--572.

\bibitem[Hav36]{havila36}
\bysame, \emph{On the momentum problem for distribution functions in more than
  one dimension {II}}, Amer.\ J.\ Math. \textbf{58} (1936), 164--168.

\bibitem[Hil88]{hilbert88}
D.~Hilbert, \emph{\"{U}ber die {D}arstellung definiter {F}ormen als {S}umme von
  {F}ormenquadraten}, Math.\ Ann. \textbf{32} (1888), 342--350.

\bibitem[HK14]{henrion14}
D.~Henrion and M.~Korda, \emph{Convex computation of the region of attraction
  of polynomial control systems}, IEEE Trans.\ Aut.\ Control \textbf{59}
  (2014), 297--312.

\bibitem[Hoi92]{hoischen92}
L.~Hoischen, \emph{A {G}eneralization of {T}heorems of {E}idelheit and
  {C}arleman {C}oncerning {A}pproximation and {I}nterpolation}, J.~Approx.\
  Theo. \textbf{71} (1992), 154--174.

\bibitem[Hor77]{horn77}
R.~A. Horn, \emph{On the {M}oments of {C}omplex {M}easures}, Math.~Z.
  \textbf{156} (1977), 1--11.

\bibitem[Hu62]{hu62}
M.-K. Hu, \emph{Visual pattern recognition by moment invariants}, IRE Trans.\
  Inf.\ Theory \textbf{12} (1962), 179--187.

\bibitem[HvNVW17]{hytonen17AnalysisBanachSpaces2}
T.~Hyt\"{o}nen, J.~van Neerven, M.~Veraar, and L.~Weis, \emph{Analysis in
  {B}anach {S}paces}, vol.~2, Springer, Cham, Switzerland, 2017.

\bibitem[Kem68]{kemper68}
J.~H.~B. Kemperman, \emph{The {G}eneral {M}oment {P}roblem, a {G}eometric
  {A}pproach}, Ann.\ Math.\ Stat. \textbf{39} (1968), 93--122.

\bibitem[Kem87]{kemper87}
\bysame, \emph{Geometry of the {M}oment {P}roblem}, Proc.\ Sym.\ Appl.\ Math.
  \textbf{37} (1987), 16--53.

\bibitem[KN77]{kreinMarkovMomentProblem}
M.~G. Kre\u{\i}n and A.~A. Nudel'man, \emph{The {M}arkow {M}oment {P}roblem and
  {E}xtremal {P}roblems}, American Mathematical Society, Providence, Rhode
  Island, 1977, translation of the Russian original from 1973.

\bibitem[Kow84]{kowalski84}
M.~A. Kowalski, \emph{A {N}ote on the {G}eneral {M}ultivariate {M}oment
  {P}roblem}, Constr.\ Theo.\ Funct. (1984), 493--499.

\bibitem[Kre70]{krein70}
M.~G. Kre\u{\i}n, \emph{The description of all solutions of the truncated power
  moment problem and some problems of operator theory}, Amer.\ Math.\ Soc.\
  Transl. \textbf{95} (1970), 219--234.

\bibitem[KS66]{karlinStuddenTSystemsBook}
S.~Karlin and W.~J. Studden, \emph{Tchebycheff {S}ystems: {W}ith {A}pplications
  in {A}nalysis and {S}tatistics}, John Wiley \& Sons, Interscience Publishers,
  New York, NY, 1966.

\bibitem[Lan80a]{landau80}
H.~J. Landau, \emph{The {C}lassical {M}oment {P}roblem: {H}ilbertian {P}roofs},
  J.~Funct.\ Anal. \textbf{38} (1980), 255--272.

\bibitem[Lan80b]{landauMomAMSProc}
H.~J. Landau (ed.), \emph{Moments in {M}athematics}, Proceedings of Symposia in
  applied Mathematics, vol.~37, Providence, RI, American Mathematical Society,
  1980.

\bibitem[Las15]{lasserreSemiAlgOpt}
J.-B. Lasserre, \emph{An introduction to polynomial and semi-algebraic
  optimization}, Cambridge University Press, Cambridge, 2015.

\bibitem[Lau09]{lauren09}
M.~Laurent, \emph{Sums of squares, moment matrices and optimization over
  polynomials}, Emerging application of algebraic geometry, IMA Vol. Math.
  Appl., vol. 149, Springer, New York, 2009, pp.~157--270.

\bibitem[LO77]{linnik77}
J.~V. Linnik and I.~V. Ostrovskii, \emph{Decomposition of {R}andom {V}ariables
  and {V}ectors}, American Mathematical Society, Providence, Rhode Island,
  1977, translation of the Russion orginal (1972).

\bibitem[Lor86]{lorentz86}
G.~G. Lorentz, \emph{Bernstein {P}olynomials}, AMS Chelsea Publishing,
  Providence, Rhode Island, 1986.

\bibitem[LPHT08]{lasserre08}
J.-B. Lasserre, C.~Prieur, D.~Henrion, and E.~Tr\'elat, \emph{Nonlinear optimal
  control via occupation measures and {LMI}-relaxations}, SIAM J.\ Control
  Optim. \textbf{47} (2008), 1649--1666.

\bibitem[LR82]{lee82}
Y.~T. Lee and A.~A.~G. Requicha, \emph{Algorithms for computing the volume and
  other integral properties of solids. {I}. known methods and open issues},
  Comm.\ ACM \textbf{25} (1982), 635--641.

\bibitem[Mar08]{marshallPosPoly}
M.~Marshall, \emph{Positive {P}olynomials and {S}ums of {S}quares},
  Mathematical Surveys and Monographs, no. 146, American Mathematical Society,
  Rhode Island, 2008.

\bibitem[MMR05]{martin05}
J.-M. Martin, K.~Mengersen, and C.~P. Robert, \emph{Bayesian modelling and
  inference on mixtures of distributions}, Handbook of Statistics \textbf{25}
  (2005), 459--507.

\bibitem[MN68]{manas68}
M.~Ma\v{n}as and J.~Nedoma, \emph{Finding all vertices of a convex polyhedron},
  Numer.\ Math. \textbf{12} (1968), 226--229.

\bibitem[MR80]{mathei80}
T.~H. Matheiss and D.~S. Rubin, \emph{A survey and comparison of methods for
  finding all vertices of convex polyhedral sets}, Math.\ Oper.\ Res.
  \textbf{5} (1980), 167--185.

\bibitem[MVKW95]{milanf95}
P.~Milanfar, G.~Verghese, W.~Karl, and A.~Willsky, \emph{Reconstructing
  polygons from moments with connections to array processing}, IEEE Trans.\
  Signal Proc. \textbf{43} (1995), 432--443.

\bibitem[MWHL20]{marx20}
S.~Marx, T.~Weisser, D.~Henrion, and J.~Lasserre, \emph{A moment approach for
  entropy solutions to nonlinear hyperbolic {PDE}s}, Math.\ Control Relat.\ F.
  \textbf{10} (2020), 113--140.

\bibitem[Pea94]{pearson94}
K.~Pearson, \emph{Contributions to the mathematical theory of evolution},
  Phil.\ Trans.\ Roy.\ Soc.\ London~A \textbf{185} (1894), 71--110.

\bibitem[P{\'{o}}l38]{polya38}
G.~P{\'{o}}lya, \emph{Sur l'ind\'etermination d'un th\'eoreme voisin du
  probl\'eme des moments}, C.~R.\ Acad.\ Sci.\ Paris \textbf{207} (1938),
  708--711.

\bibitem[Ric57]{richte57}
H.~Richter, \emph{Parameterfreie {A}bsch\"atzung und {R}ealisierung von
  {E}rwartungswerten}, Bl.\ Deutsch.\ Ges.\ Versicherungsmath. \textbf{3}
  (1957), 147--161.

\bibitem[Sat99]{sato99}
K.-I. Sato, \emph{L\'evy {P}rocesses and {I}nfinitely {D}ivisible
  {D}istributions}, Cambridge University Press, Cambridge, UK, 1999.

\bibitem[Sch79]{schmud79}
K.~Schm\"{u}dgen, \emph{A positive polynomial which is not a sum of squares.
  {A} positive, but not strongly positive functional}, Math.\ Nachr.
  \textbf{88} (1979), 385--390.

\bibitem[Sch17]{schmudMomentBook}
\bysame, \emph{The {M}oment {P}roblem}, Springer, New York, 2017.

\bibitem[Sch24]{schmud24signedArxiv}
\bysame, \emph{On the {M}oment {F}unctionals with {S}igned {R}epresenting
  {M}easures}, arXiv:2405.08899.

\bibitem[She64]{sherman64}
T.~Sherman, \emph{A {M}oment {P}roblem on $\rset^n$}, Rend.\ Circ.\ Mat.\
  Palermo \textbf{13} (1964), 273--278.

\bibitem[Sim98]{simon98}
B.~Simon, \emph{The {C}lassical {M}oment {P}roblem as a {S}elf-{A}djoint
  {F}inite {D}ifference {O}perator}, Adv.\ Math. \textbf{137} (1998), 82--203.

\bibitem[SMD{\etalchar{+}}07]{sommer07}
I.~Sommer, O.~M\"uller, F.~S. Domingues, O.~Sander, J.~Weickert, and
  T.~Lengauer, \emph{Moment invariants as shape recognition technique for
  comparing protein binding sites}, Bioinformatics \textbf{23} (2007),
  3139--3146.

\bibitem[SS03]{stein03}
E.~M. Stein and R.~Shakarchi, \emph{Fourier {A}nalysis: A {I}ntroduction},
  Princeton Lectures in Analysis, vol.~1, Princeton University Press, Princeton
  and Oxford, 2003.

\bibitem[Sti94]{stielt94}
T.~J. Stieltjes, \emph{Recherches sur les fractions continues}, Ann.~Fac.\ Sci.
  Toulouse \textbf{8} (1894), no.~4, J1--J122.

\bibitem[Sto16]{stoyan16}
J.~Stoyanov, \emph{Moment properties of probability distributions used in
  stochastic financial models}, Recent Advances in Financial Engineering 2014
  Proceedings of the TMU Finance Workshop 2014, World Scientific Publishing Co.
  Pte. Ltd., 2016, pp.~1--27.

\bibitem[Str65]{strassen65}
V.~Strassen, \emph{The existence of probability measures with given
  marginales}, Ann.\ Math.\ Stat. \textbf{36} (1965), 423--439.

\bibitem[TSM85]{titter85}
D.~M. Titterington, A.~F.~M. Smith, and U.~E. Makov, \emph{{S}tatistical
  {A}nalysis of {F}inite {M}ixture {D}istributions}, John Wiley \& Son,
  Chichester, 1985.

\bibitem[Wid72]{widder72}
D.~V. Widder, \emph{The {L}aplace {T}ransform}, Princeton University Press,
  Princeton, 1972, 8th printing.

\bibitem[Wil91]{williams91probMartingales}
D.~Williams, \emph{Probability with {M}artingales}, Cambridge University Press,
  Cambridge, 1991.

\end{thebibliography}

\newcommand{\etalchar}[1]{$^{#1}$}
\providecommand{\bysame}{\leavevmode\hbox to3em{\hrulefill}\thinspace}
\providecommand{\MR}{\relax\ifhmode\unskip\space\fi MR }
% \MRhref is called by the amsart/book/proc definition of \MR.
\providecommand{\MRhref}[2]{%
  \href{http://www.ams.org/mathscinet-getitem?mr=#1}{#2}
}
\providecommand{\href}[2]{#2}

\end{document}